\documentclass[12pt]{amsart}
\usepackage{amssymb}
\usepackage[all]{xy}
\usepackage{amsthm}
\usepackage{epsfig,graphicx}

\newtheorem{theorem}{Theorem}[section]
\newtheorem{lemma}[theorem]{Lemma}
\newtheorem{proposition}[theorem]{Proposition}
\newtheorem{prop}[theorem]{Proposition}

\newtheorem{coro}[theorem]{Corollary}
\newtheorem{lem}[theorem]{Lemma}

\theoremstyle{definition}
\newtheorem{defn}[theorem]{Definition}
\newtheorem{definition}[theorem]{Definition}

\theoremstyle{remark}
\newtheorem{rem}[theorem]{Remark}

\newtheorem{exa}[theorem]{Example}

\newcommand{\overbar}[1]{\mkern 1.5mu\overline{\mkern-1.5mu#1\mkern-1.5mu}\mkern 1.5mu}
\newcommand{\I}{\mathbb I}      
\newcommand{\Hy}{ \mathbb{H}^3 }
\newcommand{\bHy}{ \mathbb{\overline H}^3 }
\newcommand{\Cp}{\mathbb{CP}^3}
\newcommand{\C}{\mathbb{C}}
\newcommand{\R}{\mathbb{R}}
\newcommand{\limtrop}{\operatorname{lim}_{\operatorname{trop}}}
\newcommand{\ctor}{({\mathbb C}^\times)^n}

\newcommand{\tor}{(\C^\times)^{n}}

\newcommand{\dd}{\partial}
\newcommand{\am}{\mathcal{A}}
\newcommand{\coam}{\mathcal{B}}

\newcommand{\Z}{\mathbb{Z}}

\newcommand{\cp}{{\mathbb C}{\mathbb P}}
\newcommand{\rp}{{\mathbb R}{\mathbb P}}

\newcommand{\pp}{{\mathbb P}}

\newcommand{\Log}{\operatorname{Log}}

\renewcommand{\setminus}{\smallsetminus}

\newcommand{\Sym}{\operatorname{Sym}}

\newcommand{\camc}{\overline{\mathcal{A}}_C}

\newcommand{\ignore}[1]{\relax}

\begin{document}

\title{
Non-commutative amoebas}
\author{Grigory Mikhalkin and Mikhail Shkolnikov}
\address{Universit\'e de Gen\`eve,  Math\'ematiques, rue du Conseil-G\'en\'eral 7-9, 1205 Gen\`eve, Suisse}

\ignore{
\begin{abstract}
The group of isometries $\I$ of the hyperbolic space $\Hy$ is a 3-dimensional complex variety $\operatorname{PSL}_2(\C)$
which fibers over $\Hy$.
A {\em hyperbolic amoeba} is an image of an algebraic subvariety $V\subset\I$ under this fibration.
The paper surveys basic properties of the hyperbolic amoebas and compares them against more conventional
amoebas inside $\R^n$.
\end{abstract}
}

\begin{abstract}
The group of isometries of the hyperbolic space $\Hy$ is 
the 3-dimensional
group $\operatorname{PSL}_2(\C)$, which is one of the simplest non-commutative complex Lie groups. 
Its quotient by the subgroup $\operatorname{SO}(3)\subset\operatorname{PSL}_2(\C)$ naturally maps it back to $\Hy$.
Each fiber of this map is  diffeomorphic to the real projective 3-space $\rp^3$. 

The resulting map $\operatorname{PSL}_2(\C)\to\Hy$ can be viewed as the simplest non-commutative counterpart of the map $\operatorname{Log}:(\C^\times)^n\to\R^n$ from the commutative complex Lie group $(\C^\times)^n$ with the Lagrangian torus fibers that can considered as a Liouville-Arnold type integrable system. 
Gelfand, Kapranov and Zelevinsky \cite{GKZ} have introduced {\em amoebas} of algebraic varieties $V\subset(\C^\times)^n$ as images $\Log(V)\subset\R^n$.
We define the amoeba of an algebraic subvariety
of $\operatorname{PSL}_2(\C)$ as its image in $\Hy$.
The paper surveys basic properties of the resulting hyperbolic amoebas and compares them against the commutative amoebas $\R^n$.
\end{abstract}

\thanks{Research is supported in part by the grants 182111 and 178828
of the Swiss National Science Foundation.}

\maketitle

\section{Introduction}
\subsection{Three-dimensional hyperbolic space}
The hyperbolic space $\Hy$
is a complete contractible 3-space
enhanced with a Riemannian metric of constant curvature $-1$. Such a space is unique up
to isometry.

In the disk (Poincar\'e) model we
may represent $\Hy$ as the interior of the unit 3-disk
$$D^3=\{x\in\R^3\ |\ ||x||^2< 1\}\subset\R^3.$$
The geodesics on $\Hy$ in this model are cut by planar circles (or straight lines) orthogonal to
the boundary 2-sphere $\dd D^3\subset\R^3$. This 2-sphere $\dd D^3$ is called the {\em absolute}.
Similarly, geodesic 2-planes in $\Hy$ are cut by 2-spheres (or planes) in $\R^3$ perpendicular to
$\dd D^3$.

The absolute $\dd\Hy$ can be constructed intrinsically
and independent of the choice of model. Let us choose a point $x\in\Hy$.
The set of geodesic rays emanating from $x$ can be identified with $\dd\Hy$
by tracing the endpoint of the geodesic ray. On the other hand it can also be
identified with the unit tangent 2-sphere $UT(x)\approx S^2$ of the unit tangent vectors at $x$.
This gives us a canonical identification 
$$UT(x)=\dd\Hy$$
for any $x\in\Hy$ and, in particular, an identification $UT(x)=UT(y)$ for $x,y\in\Hy$.
It is easy to see that while this identification does not preserve the metric induced from
the tangent space $T\Hy$, it does preserve the conformal structure associated to that metric.
Thus the absolute $\dd\Hy$ comes with a natural conformal structure and can be identified
with the Riemann sphere
$$\dd\Hy=\cp^1.$$

The Riemann sphere $\cp^1$ is obtained from the 2-dimensional vector space $\C^2$ by projectivization.
The group $\operatorname{GL}_2(\C)$ of linear transformations of $\C^2$ acts also on $\cp^1$ so that a matrix
$\begin{pmatrix}  a & b \\ c & d \end{pmatrix}$ acts by the so-called M\"obius transformations
$$z\mapsto \frac{az+b}{cz+d}$$
on $\cp^1=\C\cup\{\infty\}$.

\begin{theorem}[Classical]
\label{pgl2c}
Any isometry of $\Hy$ extends to the absolute $\dd \Hy$.
Furthermore, 
there is a natural 1-1 correspondence between the group $\I$ of orientation-preserving 
isometries of the hyperbolic space $\Hy$  and 
the group $\operatorname{PSL}_2(\C)$ of the M\"obius transformations of the absolute $\dd \Hy\approx\cp^1$.
In other terms,
$$\I=\operatorname{PSL}_2(\C)=\operatorname{PGL}_2(\C).$$
\end{theorem}

\subsection{Groups $\I=\operatorname{PSL}_2(\C)$ and $\tilde\I=\operatorname{SL}_2(\C)$.}
The group of all M\"obius transformation can be identified with the projectivization $\operatorname{PGL}_2(\C)$ of the
linear group $\operatorname{GL}_2(\C)$ --  two linear transformation induce the same linear map if and only if one is a scalar
multiple of another. Note that as the complex numbers are algebraically closed,
the projectivization $\operatorname{PGL}_2(\C)$ of the general linear group $\operatorname{GL}_2(\C)$ coincides 
with the projectivization $\operatorname{PSL}_2(\C)=\operatorname{SL}_2(\C)/\{\pm 1\}$ of the group $\operatorname{SL}_2(\C)$ 

The group $\operatorname{SL}_2(\C)$ is the first non-trivial example of a complex simple Lie group. Its maximal compact
subgroup is the group of special unitary matrices $\operatorname{SU}(2)$. We have a diffeomorphism $\operatorname{SU}(2)\approx S^3$
as $\operatorname{SU}(2)$ acts transitively and without fixed points on the unit sphere $S^3$ in $\C^2$. We also have
a diffeomorphism 
$$\operatorname{SL}_2(\C)\approx T_*(\operatorname{SU}(2))\approx S^3\times\R^3$$
identifying $\operatorname{SL}_2(\C)$ with the tangent space $T_*(\operatorname{SU}(2))$ to its maximal compact subgroup.

In particular, $\operatorname{SL}_2(\C)$ is simply connected and coincides with the universal covering 
$\tilde\I$ of the group $\I=\operatorname{PSL}_2(\C)$, while $\I$ is the quotient of $\tilde\I$ by its center
(isomorphic to $\Z_2=\{\pm1\}$). Topologically we have 
\begin{equation}
\label{rp3}
\I\approx \rp^3\times\R^3.
\end{equation}
Both spaces, $\I$ and its universal covering $\tilde\I$ will be important fur us as 
ambient spaces. They are complex 3-folds and contain different subvarieties,
particularly, curves and surfaces. The goal of this paper is to look at geometry of these
subvarieties in the context of hyperbolic geometry, provided by presentation of $\I$
as the group of isometries of $\Hy$.

\subsection{Compactifications $\overline\I$ and $\hat\I$ of the 3-folds $\I$ and $\tilde\I$}
Theorem \ref{pgl2c} provides a convenient way to compactly $\I$. 
Indeed, the group $\I$ is identified with the non-degenerate $2\times 2$ matrices
after their projectivization. A matrix is degenerate, if its determinant is zero, i.e.
if the quadratic homogeneous polynomial $ad-bc$ vanishes. We get the following
proposition.
\begin{proposition}
We have
$$\I=\cp^3\setminus Q,$$
where $Q$ is a smooth quadric $\{ad-bc=0\}$ in the projective space $\cp^3\ni [a:b:c:d]$.
We also use notation $\dd\I=Q$. 
\end{proposition}
Note that we can recover the computation $\pi_1(\I)=\Z_2$ also from this proposition
as $\I$ is a complement of a smooth quadric in $\pp^3$. We set
$$\overline\I=\cp^3\supset\I$$
and use this space as the compactification of the group $\I$.
This viewpoint justifies the notation $\dd\I=Q$.

In its turn, the group $\tilde\I=\operatorname{SL}_2(\C)$ of $2\times 2$-matrices with determinant 1 is tautologically
identified with the affine quadric $$\tilde\I=\{(a,b,c,d)\in\C^4\ |\ ad-bc=1\}.$$
Its topological closure $\hat\I$ in $\cp^4\supset\C^4$ is a smooth 3-dimensional
projective quadric.
We use $\hat\I$ as the compactification of the (simply connected) group $\tilde\I$.

Note that the projectivization $\cp^3$ of $\C^4$ can be identified with the infinite
hyperplane $\cp^3_\infty=\cp^4\setminus\C^4$. 
Thus we may identify $\overline\I=\cp^3_\infty$.

\begin{proposition}
$\hat\I\cap \cp^3_\infty=Q=\dd\I\subset\overline\I$.
\end{proposition}
\begin{proof}
When we pass from $\C^4$ to $\cp^4$ we introduce a new coordinate that vanishes on $\cp^3_\infty$.
For an equation of degree $d$ in affine coordinates $a,b,c,d\in\C$ all monomials of order lower
than $d$ vanish when we approach $\cp^3_\infty$. In particular, the affine equation $ad-bc=1$ in $\C^4$
becomes a homogeneous equation $ad-bc=0$ in $\cp^3_\infty$.
\end{proof}

Note that $0\notin \tilde\I\subset\C^4$, so the central projection from $0$ defines
a map $\pi$ from the projective quadric $\hat\I$ to the infinite hyperplane $\overline\I$.
\begin{proposition}
The map
$$\pi:\hat\I\to\overline\I$$
is a double covering ramified along $Q\subset\overline\I$.
\end{proposition}
\begin{proof}
We have $\pi(a,b,c,d)=\pi(-a,-b,-c,-d)$, thus $\pi^{-1}([a:b:c:d])$ consists of two
points, unless $[a:b:c:d]\in Q$.
\end{proof}

\subsection{Map $\varkappa:\I\to\Hy$}
Let us fix the origin point $0\in\Hy$. As $\I$ is the group of isometries of $\Hy$
we can define the map
\begin{equation}
\label{varkappa}
\varkappa:\I\to\Hy
\end{equation}
by $$\I\ni z\mapsto z(0)\in\Hy.$$

\begin{proposition}
The map \eqref{varkappa} is a proper submersion. We have $\varkappa^{-1}(x)\approx \rp^3$, $x\in\Hy$.
\end{proposition}
\begin{proof}
The fiber $\varkappa^{-1}(0)$ consists of isometries of $\Hy$ preserving the origin $0\in\Hy$.
This group coincides with the group $\operatorname{SO}(3)\approx\rp^3$ of isometries of the tangent space $T_0(\Hy)$.
As $\Hy$ is a homogeneous space for the group $\I$, the same holds for any other fiber $\varkappa^{-1}(x)$.
\end{proof}

\section{Amoebas and coamoebas in $\Hy$}
\subsection{Amoebas}
Let $V\subset\I$ be an algebraic subvariety. This means that  $V=\overline{V}\setminus {\dd\I}$ for  a projective subvariety $\overline{V}\subset\overline{I}=\cp^3$. Without loss of generality we may assume that $\overline{V}$ is the closure of $V$ in $\overline{\I}$.
\begin{definition}
The amoeba $$\am=\varkappa(V)\subset\Hy$$  is the image of $V$ under
the map $\varkappa$.
\end{definition}

This definition can be thought of as a hyperbolic (or non-commutative) counterpart of the amoebas
of varieties in $\tor$ defined in \cite{GKZ}. These amoebas are the primary geometric objects 
studied in this paper.

The map $\varkappa$ can be extended to the compactification $\overline\I\supset\I$ once
we extend the target $\Hy$ to its own compactification $$\bHy=\Hy\cup\dd\Hy.$$
Recall that an element of $\overline\I=\cp^3$ is a non-zero matrix 
$\begin{pmatrix}  a & b \\ c & d \end{pmatrix}$ up to multiplication by a scalar.
If a matrix is non-degenerate (rank 2) it is an element of $\I$, while
a degenerate non-zero (rank 1) matrix is an element of $\dd\I=\overline\I\setminus\I$.

A matrix $z\in\dd\I$ is a map $\C^2\to\C^2$ with a 1-dimensional kernel $A\subset\C^2$
and a 1-dimensional image $B\subset\C^2$, thus $A,B\in\cp^1$.
Note that a choice of such $A$ and $B$ uniquely determines $z\in\dd\I$
as the only remaining ambiguity is a scalar factor.
This gives us an isomorphism between the smooth projective quadric $Q=\dd\I$
and $\cp^1\times\cp^1$. Define the projections to the first and second factors
\begin{equation}\label{Qproj}
\pi_{\pm}:Q\to\cp^1
\end{equation}
by $\pi_-(z)=A$ and $\pi_+(z)=B$.

For $z\in\dd\I$ we define $\overline\varkappa(z)=\pi_+(B)$.
For $z\in\I$ we define $\overline\varkappa(z)=\varkappa(z)$.

\begin{proposition}
The map $$\overline\varkappa:\overline\I\to\bHy$$
is a continuous map from the complex projective 3-space.
\end{proposition}
\begin{proof}
If $z\in\I$ is close to $\dd\I$ then one of the two eigenvalues of the corresponding
unimodular linear map $\C^2\to\C^2$ is very small while the other is very large.
Let us mark the points $p,q$ corresponding to these eigenspaces on the
absolute $\dd\Hy=\cp^1$. The projective-linear transformation has these two points
as its fixed points. 

Since the eigenvalue corresponding to $q$ is very large, a
small neighbourhood of $q$ will contain the image of almost entire
$\dd\Hy$ (except for a small neighborhood of $p$). The same holds for the
extension of our projective-linear transformation to $\bHy$.
Thus the image of the origin will be contained in a small neighborhood of $q$,
so that $\overline\varkappa$ is continuous.
\end{proof}

\begin{coro}\label{closed-set}
The amoeba $\varkappa(V)\subset\Hy$ is a closed set.
\end{coro}

\begin{definition}
For a subvariety $\overline V\subset\overline\I\approx\cp^3$ we 
define its compactified amoeba $$\bar\am\subset\bHy$$
as the image of $\overline V$ under $\overline\varkappa$.
\end{definition}

The following proposition is straightforward.
\begin{prop}\label{l-act}
The left action of $A\in\I$ on $V$ translates its amoeba by the isometry $A,$ i.e. $\varkappa(A\cdot V)=A(\varkappa(V)).$ 
\end{prop}
The right action $V\mapsto V\cdot A$ can significantly changes the shape of $\varkappa(V)\subset\Hy$. There is an obvious exception: the right action on a subvariety by a rotation $A\in \operatorname{SO}(3)$ around $0\in\Hy$ doesn't change its amoeba. If $V$ is irreducible then its amoeba is connected. 

\subsection{Coamoebas}
The topological diffeomorphism \eqref{rp3} can be upgraded once we recall that 
the factor $\rp^3$ is a compact real Lie group $G=\operatorname{SO}(3)$. The bi-invariant metric
on $G$ is well-defined up to a scalar. It coincides with the spherical metric on $\rp^3$.

Furthermore, the whole group $\I$ can be recovered as the {\em complexification} $G_{\C}$ of $G$.
Geometrically, $G_{\C}$ can be identified with  the total space $T_*G$ of the tangent bundle to $G$.
Namely, the tangent space to the unit element $1\in G$ is the Lie algebra $\mathfrak g$ of $G$
and can be thought of as infinitely small elements of $G$. Any element of $G_{\C}=T_*G$
can be represented as a product of an element of $G$ and an element of $\mathfrak g$.
The total space $G_{\C}$ can be given complex and group structures and comes with
the map $\iota:G_{\C}\to G$.

In algebraic terms, with the help of polar decomposition of matrices
we can uniquely write $\tilde{z}=\tilde{u}p$ for any $\tilde z\in\tilde\I$,
where $\tilde{u}$ is a unitary matrix and $p$ is a non-negatively definite hermitian matrix. 
Up to sign the matrix $\pm\tilde z$ gives an element of $\I$.
Its polar decomposition defines the unitary matrix $\pm\tilde u$ up to sign,
which in its turn can be considered as an element of $G$.
We define $\iota(\pm\tilde z)=\pm\tilde u$ and thus 
\begin{equation}
\label{iota}
\iota:\I=G_{\C}\to G=\operatorname{SO}(3)\approx\rp^3
\end{equation}
is a continuous map.

\begin{definition}
For a subvariety $V\subset\I$ we 
define its coamoeba $$\coam\subset G\approx\rp^3$$
as the image $\coam=\iota(V)$.
\end{definition}

\ignore{
We finish this subsection by interpreting the map $\iota$ in terms of hyperbolic geometry.
The Lie algebra $\mathfrak g$ of $G$ consists of skew-hermitian matrices. The exponent
map takes $\mathfrak g$ to $G$. The complexified Lie algebra decomposes
as $\mathfrak g\oplus i\mathfrak g$. The image of $ i\mathfrak g$ under the exponent complexification map
coincides with the fiber over $1\in G$ for the map $G_{\C}\to G$.

Let $v\in i\mathfrak g$.
The matrix $p=\exp(v)$ is positive-definite hermitian.
The vector $v$ also can be interpreted as a tangent vector to $0\in\Hy$ as
$v/i\in\mathfrak g$ is an infinitesimal rotation along an oriented real axis in $\R^3=T_0\Hy$.
The isometry of $\Hy$ corresponding to $p$ is the extension of
the parallel transport along the geodesic
path in the exponential image of the interval $[0,v]$. In particular,
$\varkappa(p)=\exp(v)$, where the last $\exp$ is taken in the sense of Riemannian geometry.
}

Let $z\in\I$ be an arbitrary isometry of $\Hy$ with $\varkappa(z)=x$.
A parallel transport along a geodesic path in $\Hy$ connecting $0$ and $x$ 
provides a preferred isometry between tangent spaces $T_0\Hy$ and $T_x\Hy$
and thus an element $p\subset\I$ corresponding to a unimodular positive-definite hermitian
matrix with $\varkappa(p)=x$. The isometry $z$ can be obtained by taking a composition of $p$
with a self-isometry of $T_0\Hy$ which corresponds to an orthogonal matrix $u\in G$.
Thus we recover polar decomposition in hyperbolic geometry terms.

\section{Amoebas of curves}
%
\subsection{Amoebas of lines}
The shape of the amoeba ${\am_l}=\varkappa(l\cap\I)$ of a line $l\subset\Cp=\I\cup Q$ depends on the position of $l$ with respect to the quadric $Q$: either $l$ lies on $Q$, is tangent to $Q$ or intersects it transversally in two points.   
 
In the case when the line lies on the quadric, the amoeba and the intersection of $l$ with $\I$ are both empty. To describe the image of $l$ under $\overbar\varkappa$ consider the same identification of $Q$ with $\mathbb{CP}^1\times\mathbb{CP}^1$ given by \eqref{Qproj}.
There are exactly two families of lines in $Q$ appearing as fibers of
$\pi_+$ and $\pi_-$.
If the line $l$ is a fiber of $\pi_+$ then
its image is a single point.
If $l$ is a fiber of $\pi_-$ then
$\overbar\varkappa$ projects $l$ isomorphically to $\partial\Hy.$

If $Q$ doesn't contain $l$ then they intersect either at one or two points. This cases give amoebas of quite different shape. Consider first a line $l$ which meets the quadric exactly at one point, i.e. $l$ is tangent to $Q$. This implies that the amoeba of $l$ is non-empty and touches $\partial\Hy$ at a single point. 

Recall that a horosphere is a surface in $\Hy$ such that it is orthogonal to any geodesic starting at some fixed point at the infinity $\dd\Hy$.

\begin{proposition}\label{p_hor}
If a line is tangent to $Q$ then its hyperbolic amoeba is a horosphere in $\Hy.$
Conversely, any horosphere in $\Hy$ is an amoeba of a line tangent to $Q$.
\end{proposition}

One can give the following equivalent statement for this proposition. A hyperbolic amoeba of a line in the 3-dimensional quadric $\tilde\I$ is a horosphere. Indeed, a curve in $\tilde\I$ is a line if and only if 
its image under
the two-fold covering $\tilde\I\rightarrow\I$
is a line tangent to $Q$.
%
We also note that both families of lines on $Q$ 
give amoebas, which can be interpreted as infinitely small and infinitely large horospheres. 
The proof of this proposition is given after the proof of Lemma \ref{lem-l2}.

An amoeba of a line transverse to $Q$ is generic. It must be nonempty, with two infinite points at its closure. 
A cylinder of radius $r\geq 0$ in $\Hy$ is defined to be a locus of points that are at the same distance $r$ from a given geodesic.  The degenerate cylinder for $r=0$ coincides with the geodesic itself.

\begin{proposition}\label{p_cyl}
If a line is not tangent to $Q$ then its hyperbolic amoeba is a (possibly degenerate) cylinder in $\Hy.$
Conversely, any geodesic in $\Hy$ as well as any cylinder of radius $r>0$ is the amoeba of a line in $\overline{\I}$.
\end{proposition}




Both left and right actions of $\I$ on itself can be uniquely extended to $\overline{\I}=\Cp$. Clearly, $Q$ is invariant under these actions. 
\begin{lem}\label{lr-act}
The left action of $\I$ on $Q=\cp^1\times\cp^1$ acts by M\"obius transformations on the second factor and conserves the first factor. Namely, for $A\in I$ we have $A\cdot (\alpha,\beta)=(\alpha,A(\beta))$ where $A(\beta)\in\cp^1$ is the image of $\beta$ under the M\"obius action.

Similarly, the right action of $\I$ on $Q=\cp^1\times\cp^1$ conserves the second factor and acts on the first factor by M\"obius transformations $(\alpha,\beta)\cdot A=(A^{-1}(\alpha),\beta)$.
\end{lem}
\begin{proof}
Recall that $Q$ corresponds to rank 1 matrices $P$, the coordinate $\alpha$ is the projectivization of the 1-dimensional kernel of $P$ while $\beta$ is the projectivization of the 1-dimensional image of $P$.
\end{proof}
\begin{coro}\label{coro-lr}
For any two lines $l,l'\subset\overline{\I}$ transverse to $Q$ there exist elements $A,B\in\I$ such that $l'=AlB$. The transformations $A,B$ are unique up to (left or right) multiplication by subgroups isomorphic to $\C^\times$.
\end{coro}
\begin{proof}
The lines $l$ and $l'$ are determined by the pairs of their intersection points with $Q$. Since both lines are transverse to $Q$, each pair produces a pair of distinct points in $\cp^1$ under $\pi_+$ and $\pi_-$. M\"obius transformations act transitively on pairs of distinct points in $\cp^1$ with the stabiliser isomorphic to $\C^\times$.
\end{proof}

Consider the space of all lines in $\Cp$ tangent to $Q$. We have an action of $\I\times \operatorname{SO}(3)$ on this space, where the first factor $\I$ acts on the left and the second factor $\operatorname{SO}(3)\subset\I$ acts on the right. Lines in the same orbit of this action have congruent hyperbolic amoebas. We claim that there are only three different orbits for the action.
\begin{lemma}\label{l_orb}
Two families of lines lying on $Q$ and a set of all other lines tangent to $Q$ are the only orbits for the action of $\I\times \operatorname{SO}(3)$ on $Q.$
\end{lemma}
 Thus, if we show that an amoeba of some particular properly tangent line to $Q$ is a horosphere then amoebas of all other lines of this kind will be horosphere.    

\begin{proof}
Let $x$ be a point in $Q$. Take a stabiliser subgroup for $x$ under the action of $\I\times \operatorname{SO}(3)$ and consider its action on the tangent space to $Q$ at the point $x$. It is clear that the stabiliser has a subgroup isomorphic to $\mathbb{C}^\times \times U(1)$ and acts separately on each multiplier in $T_x Q=\mathbb{C}\times\mathbb{C}$. The action of this subgroup on the projectivization $\mathbb{P}(T_xQ)$  for the tangent space can be obviously reduced to the standard action of $\mathbb{C}^\times$ on $\mathbb{CP}^1$. The last is stratified on three orbits: two points and a torus. The points correspond to the lines in the intersection of $Q$ and a plane tangent to $Q$ at $x$. The torus parametrizes the space of all lines properly tangent to $Q$ at $x$.

To finish the prove note that $\I\times \operatorname{SO}(3)$ acts transitively on $Q$ and evidently preserves the stratification for projectivization of a tangent space to $Q$ at each point. 
\end{proof}
 
Our goal now is to show that there exist a line tangent to $Q$ with an amoeba equal to a horosphere in $\Hy$. The main idea here is to use interactions of some specific subgroups of $\I$ to produce extra symmetries for their amoebas. 

As we saw before, the group $\I$ can be interpreted as the group of automorphisms for $\mathbb{CP}^1=\mathbb{C}\cup\{\infty\}$. A group $B$ of affine transformations on the complex line $\mathbb{C}$ is a $2$-dimensional subgroup of $\I$. It can be also defined to be a stabiliser of $\infty.$ In fact $B$ can be described up to conjugation as a Borel subgroup of $\I$.  Each element of $\I$ can be seen a M\"obius transformation $$z\mapsto{{az+b}\over{cz+d}},\qquad ad-bc\neq 0.$$ Such transformation is affine if and only if $c=0$ and is given by $z\mapsto az+b$. So the closure of $B$ is a plane in $\Cp$. 
 
Consider the following two subgroups in $B$: the subgroup $l_1=\{z\mapsto z+b\}$ of translations in $\mathbb{C}$ and the subgroup $l_2=\{z\mapsto az\}$  generated by homotheties and rotations around $0\in\mathbb{C}$. In the above notations $l_1$ is given by the equation $a=1$ and $l_2$ is given by $b=0$. So the closures for both subgroups are lines in $\Cp.$  It is also clear that $l_1$ is a normal subgroup of $B$ and so $l_2$ acts on $l_1$ by the conjugation, $l_1$ is a maximal affine subgroup and $l_2$ is a maximal torus in $\I$. Since $l_1$ and $l_2$ intersect by a unity and generate the whole group of affine transformations, $B$ is a semi-direct product of $l_1$ and $l_2$.

Now we are going to describe amoebas for these groups. First we have to note that an amoeba of any subgroup in $\I$ is smooth at each point because the group acts on its amoeba transitively by isometries.  

We start with $l_2.$ The point $0$ and $\infty$ in $\partial\Hy$ are the only points stabilized by $l_2.$ Denote by $\gamma\subset\Hy$ the geodesic connecting these points.
\begin{lemma}\label{lem-l2}
We have $\am_{l_2}=\gamma$.
\end{lemma}
\begin{proof}
The elements of $l_2$ correspond to isometries of $\Hy$ fixing the geodesic $\gamma$.
\end{proof}


The only point in $\dd\Hy$ which is preserved by $l_1$ is the point $\infty$.
The amoeba $\am_{l_1}$ is the horosphere tangent to $\infty$ and passing through $0$.
This observation  completes the proof of the Proposition \ref{p_hor}.
%
%

\begin{proof}[Proof of Proposition \ref{p_cyl}]
A line $l$ not tangent to $Q$ meets the quadric at two points $(q_1,p_1),\ (q_2,p_2)\in Q.$ 
Since $l$ is not contained in $Q$, we have $p_1\neq p_2\in\cp^1$ and $q_1\neq q_2\in\cp^1$.
Thus there exist M\"obius transformations sending pairs $(p_1,p_2)$ and $(q_1,q_2)$ to $(0,\infty)$.
Note that the line $l_2$ meets $Q$ at the points $(0,0)$ and $(\infty,\infty)$.
By Corollary \ref{coro-lr} there exist $A,B\in\I$ such that $l=Al_2B$. Since $l_2$ consist of isometries of $\Hy$ fixing $\gamma$, the amoeba $\am_{l_2B}$ is a cylinder around $\gamma$ of radius equal to the distance between $B(0)$ and $\gamma$. By Proposition \ref{l-act}, $\am_l$ is a cylinder obtained as the image of $\am_{l_2B}$ under the isometry $A$.
\end{proof}

\begin{prop}\label{line-descr}
Let $l$ be the line connecting $(q_1,p_1)$ and $(q_2,p_2)\in \cp^1\times\cp^1=Q$ with $p_1\neq p_2\in\cp^1=\C\subset\{\infty\}$ and $q_1\neq q_2,\in\cp^1=\C\subset\{\infty\}$.
\begin{enumerate}
\item
The amoeba  $\am_l$ is a geodesic line if and only if $$q_2=-1/\bar{q}_1.$$
In this case $\varkappa|_{l\cap\I}$ is a circle bundle over its image.

Conversely, for each geodesic line $\gamma\subset\Hy$ and a point $x\in\varkappa^{-1}(\gamma)$ there exists a unique line $l\ni x$ such that $\am_l=\gamma$.

\item
If $q_2=-1/\bar{q}_1$ then $\am_l$ is a non-degenerate cylinder whose radius depends only on the spherical distance between $q_2$ and $-1/\bar{q}_1$ (the antipodal point to $q_1\in\cp^1$). Dependance of the radius of this distance is monotone and goes to infinity when $q_2$ approaches $q_1$.
In this case $\varkappa|_{l\cap\I}$ is a smooth proper embedding to $\Hy$.

Conversely, for any non-degenerate cylinder $Z\subset\Hy$ and a point $x\in\varkappa^{-1}(Z)$ there exists a line $l\ni x$ such that $\am_l=Z$. 
\end{enumerate}
\end{prop}
The spherical distance on $\cp^1\approx S^2$ is the distance in the metric preserved by the subgroup $\operatorname{SO}(3)\subset\I$ acting on $\cp^1$ by M\"obius transformations.
\begin{proof}
By the proof of Lemma \ref{lem-l2}, $\varkappa(l_2B)=\gamma$ if and only if $B\in\varkappa^{-1}(\gamma)$.
This is the case when the right action of $B$ decomposes to a transformation of $\Hy$ preserving the origin and a transformation of $\Hy$ preserving $\gamma$. Since the subgroup $\operatorname{SO}(3)\subset\I$ 
consists of isometries of $\Hy$ preserving the origin, the amoeba  $\varkappa(l)$ is a geodesic
if and only if $q_1$ and $q_2$ are antipodal points in the spherical metric on $\cp^1$, i.e. $q_2=-1/\bar{q}_1.$ 

By Proposition \ref{p_cyl} any cylinder in $\Hy$ is an amoeba of a line passing through $(q_1,p_1), (q_2,p_2)\in Q$. Multiplying this line by an appropriate element of $\operatorname{SO}(3)$ on the right we may ensure that $l$ contains any given point over its amoeba.
Distinct lines passing through $(q_1,p_1)$ cannot intersect at a point $x\neq (q_1,p_1)$. This implies uniqueness of a line with the same amoeba up to the right action of $\operatorname{SO}(3)$. In its turn, the uniqueness implies monotonicity of the cylinder radius as a function of spherical distance between $q_1$ and $q_2$.
\end{proof}

\subsection{Gauss maps $\gamma_\pm$} 
Let $\overline{C}\subset\overline{\I}=\cp^3$ be an irreducible spatial projective curve of degree $d$ and genus $g$ not contained in $\dd\I=Q$.
Define $C=\overline{C}\setminus\dd\I\subset\I$ and $\am_C=\varkappa(C\cap\I)\subset\Hy$.
Since $Q\subset\cp^3$ is a quadric, we get the following proposition for the compactified amoeba $\camc$.

\begin{prop}
The intersection $\camc\cup\dd\Hy=\overline{\varkappa}(\overline C\cap Q)$ consists of not more than $2d$ points. 
\end{prop}

We define $l_x\subset\overline{I}$ as the tangent line to $C$ at $x\in C$.
The intersection $l_x\cap Q$ is an unordered pair of (perhaps coinciding) points in $Q=\cp^1\times\cp^1$.
The projections of this pair under $\pi_-$ and $\pi_+$ define the elements
$$\gamma_-(x), \gamma_+(x)\in \operatorname{Sym}^2(\cp^1)=\cp^2.$$
By the removable singularity theorem, these maps uniquely extend to maps
\[
\gamma_-,\gamma_+:\overline C\to\cp^2
\]
called the right and left {\em Gauss maps}.

\begin{rem}\label{rem-Sym}
The maps $\gamma_\pm$ are non-commutative counterparts of the logarithmic Gauss map \cite{Ka-gauss}. They can be alternatively defined by taking the tangent direction of $C$ at the unit element $E\in\I$ after left or right translate of $C$ by $x\in\I$.
To see this, we identify the projectivization 
$
T_E(\I)\approx\cp^2
$
of the tangent space of $\I$ at the unit element $E\in\I$ with the space of lines in $\overline{\I}$ passing through $E$. The image $\pi_\pm(L\cap Q)$ can be considered as a point in 
$$
\operatorname{Sym}^2(\cp^1)=\cp^2=T_E(\I).
$$
By Lemma \ref{lr-act}, the left action conserves $\pi_+$ while the right action conserves $\pi_-$.
\end{rem}

\begin{prop}
The degree of $\gamma_\pm$ is $2d-2+2g$.
\end{prop}
\begin{proof}
Note that the image of the map $\cp^1\to\operatorname{Sym}^2(\cp^1)=\cp^2$ mapping $z\in\cp^1$ to the unordered pair consisting of $z$ and a fixed point $z_0\in\cp^1$ is 1. Thus the degree of $\gamma_\pm$ can be computed as the number of planes in $\cp^3$ in the pencil passing through a given line in $Q$ that are tangent to $C$. The proposition follows from the Riemann-Hurwitz formula for the corresponding map from $C$ to the pencil.
\end{proof}



Denote by $R\subset\cp^2$ the fixed locus of the antiholomorphic involution $$\sigma_R(u:v:w)\mapsto (\bar w:-\bar v:\bar u).$$
The holomorphic change of coordinates $(u:v:w)\mapsto (u+w:iv:i(w-u))$ identifies $\sigma_R$ with the involution of complex conjugation in $\cp^2$ and $R$ with $\rp^2$. In particular, $R$ is a totally real surface in $\cp^2$.
The following lemma is a counterpart of \cite[Lemma 3]{Mi00} for non-commutative amoebas.

\begin{lem}\label{crit-kap}
A smooth point $x\in C$ is critical for $\varkappa|_C$ if and only if $\gamma_-(x)\in R$.
\end{lem}
\begin{proof}
The point $x$ is critical for $\varkappa|_C$ if and only if it is critical for the restriction of $\varkappa$ to the tangent line $l_x\subset\overline\I$ to $C$ at $x$. By Proposition \ref{line-descr} this is determined by the type of the amoeba of $l_x$. If the amoeba of $l_x$ is a geodesic line then all points of $l_x$ are critical for $\varkappa|_{l_x}$. If the amoeba of $l_x$ is a non-degenerate cylinder then  all points of $l_x$ are regular for $\varkappa|_{l_x}$.

In the affine chart $\C^2\ni (a,b)$ corresponding to $u\neq 0$ we have the identification $\C^2=\operatorname{Sym}^2(\C)$ with $a=z+z'$, $b=zz'$ for $z,z'\in\C$. 
The antiholomorphic involution $\sigma_R$ is given by the involution $z\mapsto -1/\bar{q}_1$. By Proposition \ref{line-descr} its fixed locus corresponds to the lines whose amoebas are geodesic lines.
\end{proof}

\begin{theorem}\label{ga-crit}
Let $\Omega$ be a domain in $\mathbb{C}$ and $\phi\colon\Omega\rightarrow\I$ be a holomorphic embedding. If $\varkappa\circ\phi$ is critical at every point then $\phi$ parametrizes a part of a line in $\I$ projecting by $\varkappa$ on a geodesic. 
\end{theorem}

\begin{coro}
The amoeba $\am_C\subset\Hy$ of a non-singular irreducible curve $C\subset\cp^3$ is smoothly immersed at its generic point.
\end{coro}

\begin{proof}[Proof of Theorem \ref{ga-crit}]
By Lemma \ref{crit-kap}, if $\varkappa\circ\phi$ is critical then its image is contained in the totally real surface $R$. Thus $\varkappa\circ\phi$ is constant and $\phi(\Omega)$ is everywhere tangent to the same line.
\end{proof}

\section{Tropical limits of amoebas of curves}
\subsection{Tropical limits for
curves in $\ctor$}
Let us recall how tropical limits appear in the context of conventional (commutative) amoebas, i.e. amoebas of algebraic varieties $V\subset(\C^\times)^n$, see \cite{IMS}.
The Lie group $\ctor$ is commutative, its maximal compact subgroup is the real torus $(S^1)^n$.
The map 
\[
\Log_t:\ctor\to\R^n, \ \Log_t(z_1,\dots,z_n)=(\log_t|z_1|,\dots,\log_t|z_n|)
\]
takes the quotient by this subgroup. The image $\Log(V)\subset\R^n$ is called the amoeba of $V$ for $t=e$, using the notation $\Log=\Log_e$.

\begin{defn}
An unbounded family $\{t_\alpha\}$ of real numbers $t_\alpha>1$, $\alpha\in A$, is called a
{\em scaling sequence}. 
\end{defn}
Scaling sequences allow to define tropical limits $\limtrop$ for families of various geometric objects parameterised by $A$ (here we particularly refer to \cite{ItMi} and \cite{KaMi} for examples and details). Such families are called {\em scaled families}.
In the simplest situation, if $\{z_\alpha\}_{\alpha\in A}$ is a family of complex numbers we define
\[
\limtrop z_\alpha=\lim\limits_{t_\alpha\to\infty}\log_{t_\alpha}|z_\alpha|.
\]
This limit may be finite, i.e. an element of $\R$, infinite, i.e. $-\infty$ or $+\infty$, or not exist at all.
Since $\R\cup\{\pm\infty\}\approx [0,1]$ is compact, there exists an unbounded subfamily $B\subset A$ such that $\limtrop z_\beta\in\R\cup\{\pm\infty\}$ exists, $t_\beta\to+\infty$, $\beta\in B$. 

If $V_\alpha\subset\ctor$, $\alpha\in A$, is a family of closed (in the Euclidean topology) sets (e.g. algebraic curves of the same degree) we define
\[
\limtrop V_\alpha=\lim\limits_{t_\alpha\to\infty}\Log_{t_\alpha}(V_\alpha)\subset\R^n,
\]
where the latter limit is considered in the sense of the Hausdorff metric in all compact subsets of $\R^n$.

By an {\em open finite graph} we mean the complement $\Gamma=\bar\Gamma\setminus\dd\Gamma$ of a subset $\dd\Gamma$ of the set of 1-valent vertices of a finite graph $\bar\Gamma$. 
\begin{defn}
A tropical curve is a finite metric graph $\Gamma$ with a complete inner metric together with a choice of the {\em genus function}
$g:\operatorname{Vert}(\Gamma)\to\Z_{\ge 0}$
from the set of vertices of $\Gamma$ such that $g(v)\neq 0$ if $v$ is a 1-valent vertex.
A vertex $v$ of $\Gamma$ is called {\em essential} if its genus $g(v)$ is positive or its valence is greater than 2.
Two tropical curves are considered to be the same if there exists an isometry between them such that a vertex of positive genus corresponds to a vertex of the same genus. Clearly, essential vertices must correspond to essential vertices under such a correspondence.
\end{defn}
Recall that a metric is inner if the distance between any two points is given by the smallest length of a connecting path. Such metric carries the same information as specifying the lengths of all edges.
Completeness of the metric is equivalent to the condition that the open leaf-edges, i.e. the edges adjacent to $\dd\Gamma\subset\bar\Gamma$, have the infinite lengths.

Note that the simplest tropical curve $I\approx\R$ has a single (double-infinite) edge and no vertices. It is obtained from the interval graph (a finite graph with two vertices and a connecting edge) by removing both 1-valent vertices. A circle $E_\tau=\R/\tau\Z$ of perimeter $\tau>0$ (in the inner metric) is an example of a compact tropical curve without essential vertices. It is called a (compact) tropical elliptic curve of length $\tau$. All other connected tropical curves have at least one essential vertex.
\begin{defn}
A parameterised tropical curve in $\R^n$ is a continuous map $h:\Gamma\to\R^n$ from a tropical curve $\Gamma$ such that the following two conditions hold.
\begin{itemize}
\item
The restriction $h|_E$ to each edge $E\subset\Gamma$ is a smooth map whose differential sends the unit tangent vector to $E$ (with a chosen orientation) to an integer vector $u(E)\in\Z^n$.
\item
At every vertex $v\in\Gamma$ the {\em balancing condition}
\begin{equation}\label{balancing}
\sum\limits_E u(E)=0
\end{equation}
hold. Here the sum is taken over all edges $E$ adjacent to $v$ oriented away from $v$.
\end{itemize}
The image $Y=h(\Gamma)\subset\R^n$ can be viewed as a graph embedded to $\R^n$. For a generic point $y$ of an edge $E_y\subset Y$ the inverse image $h^{-1}(y)$ is a finite set contained inside the edges of $\Gamma$. 
We define the {\em weight} 
$$w(E_y)=\sum\limits_{x\in h^{-1}(y)} w(E_x),$$
where $E_x$ is the edge containing $x$, and $w(E_x)\in \Z_{\ge 0}$ is the largest integer number such that $u(E_x)/w(E_x)$ is integer. By the balancing condition, $w(E_y)$ does not depend on the choice of $y$.
The image $Y\subset\R^n$ with the weight data is called the {\em unparameterised tropical curve}.
\end{defn}

For an algebraic curve $V_\alpha\subset\ctor$ we may define its degree $\deg(V_\alpha)\in\Z_{\ge 0}$ as the projective degree of its closure in $\cp^n\supset\ctor$. There is a corresponding notion for tropical curves in $\R^n$.
Namely, we orient each unbounded edge $E$ of an unparameterised tropical curve $Y\subset\R^n$ toward infinity and define $\deg(E)$ to be zero if all the coordinates of $u(E)\in\Z^n$ are non-positive and to be the maximal value of these coordinates otherwise, see \cite{Mi-notices}. We define $\deg Y \in\Z_{\ge 0}$ as the sum of the degrees of all its unbounded edges.

\begin{theorem}[Unparameterised tropical curve compactness theorem
\cite{Mi05}]\label{thm-unparam}
If $V_\alpha\subset\ctor$, $t_\alpha\to\infty$, $\alpha\in A$, is a scaled family of curves of degree $d$ then there exists a scaling subsequence $t_\beta\to\infty$, $\beta\in B\subset A$, and an unparameterised tropical curve $Y\subset\R^n$ of degree $\deg Y\le d$ such that $\limtrop V_\beta=Y$.
\end{theorem}

\subsection{Phase-tropical limits over tropical curves in $\R^n$}
In this setting it is convenient to present $V_\alpha\subset\ctor$ as the image of a proper holomorphic map 
\begin{equation}\label{falpha}
f_\alpha:C_\alpha\to\ctor
\end{equation}
from a Riemann surface of finite type $C_\alpha$, $\alpha\in A$.

\newcommand{\bm}{\overline{\mathcal{M}}}
\newcommand{\bu}{\overline{\mathcal{U}}}

Suppose that $C_\alpha$ is connected for all $\alpha\in A$, and that the genus $g$ of $C_\alpha$ and the number $k$ of its punctures does not depend on $\alpha$.
By a nodal curve of genus $g$ with $k$ punctures we mean a connected but possibly reducible nodal curve, such that its topological smoothing is a (smooth) curve of genus $g$ with $k$ punctures.
Assume that $2-2g-k<0$.
A stable curve is a nodal curve such that  each of its components is either of positive genus, or is a sphere adjacent to at least three nodes or punctures.
In other words, a connected nodal curve is not stable if it contains a punctured spherical component passing through a single node or a non-punctured spherical component passing through one or two nodes.
Such components can be contracted while staying in the class of connected nodal curves and thus any nodal curve can be modified to a canonical stable curve.
Given $g$ and $k$, the space of stable curves is compact, and admits a universal curve over it, see \cite{Mum}. If we distinguish different punctures, i.e. mark them by numbers $1,\dots,k$, then the universal curve is known as $\bu_{g,k}$, it admits a continuous map
\begin{equation}\label{uncurve}
\bu_{g,k}\to\bm_{g,k}
\end{equation}
over the space of the stable curves so that the fiber over each curve $C\in\bm_{g,k}$ is this curve itself.
\ignore{
As it was noticed and exploited in \cite{KoMa}, there is a similar notion of stable parameterized curves. Namely, a parameterized proper curve $f:C\to\ctor$ from a nodal curve $C$ of genus $g$ with $k$ punctures is called stable if each component of $C$ contracted to a point by $f$ is either of positive genus, or is a sphere adjacent to at least three nodes or punctures.
In other words, we do not put any restrictions on the components of $C$ mapped nontrivially by $f$ and put the same restrictions on components contracted to points as before.
}

\ignore{
The {\em toric degree} $\delta=\{\delta_j\}_{j=1}^k$ of a map $f:C\to\ctor$ consists of integer vectors $\delta_j\in\Z^n$ given by the classes in $H_1(\ctor)=\Z^n$ of small loops going around the $j$th puncture of $C$ in the positive direction. If $f$ is proper we have $\delta_j\neq 0$. 
}

By a closed nodal curve we mean a nodal curve without punctures.
Given a (nodal) curve $C$ of genus $g$ with $k$ punctures we may consider the corresponding closed curve $\overline C$ with $k$ {\em marked points}
by filling the point at each puncture and marking it. 
The map $f$ uniquely extends to
\begin{equation}\label{barf}
\bar f:\overline C\to\cp^n
\end{equation}
which we call nodal curve of genus $g$ with $k$ marked points. Its degree is the sum of the degrees of its components.
The curve \eqref{barf} is stable 
if every component of $\overline C$ contracted to a point by $\bar f$ is either of positive genus, or is a sphere adjacent to at least three nodes or punctures.
In other words, we do not put any restrictions on the components of $C$ mapped nontrivially by $f$ and put the same restrictions on components contracted to points as before.

As it was observed in \cite{KoMa}, all stable curves of a given degree in $\cp^n$ form a compact space $\bm_{g,k}(\cp^n,d)$.
In addition we have the universal curve 
\begin{equation}\label{uncurvecpn}
\overline F:\bu_{g,k}(\cp^n,d)\to\cp^n
\end{equation}
and the map
\begin{equation}\label{mappi}
\pi:\bu_{g,k}(\cp^n,d)\to \bm_{g,k}(\cp^n,d)
\end{equation} 
such that for any $\bar f\in \bm_{g,k}(\cp^n,d)$ the inverse image $\pi^{-1}(\bar f)$ is a nodal curve, and the restriction of $F$ to $\pi^{-1}(\bar f)$ coincides with $\bar f$. 
Furthermore, we have the continuous forgetting map 
\begin{equation}\label{uft}
\tilde{\operatorname{ft}}:\bu_{g,k}(\cp^n,d)\to\bu_{g,k}
\end{equation}
mapping a stable curve \eqref{barf} in the fiber of \eqref{mappi} to the corresponding stable curve in the fiber of \eqref{uncurve} obtained from the source nodal curve  $\overline C$ after contracting all its non-stable spherical components.
In particular, it induces the forgetting map 
\begin{equation}\label{ft}
{\operatorname{ft}}:\bm_{g,k}(\cp^n,d)\to\bm_{g,k}.
\end{equation}

We use a combination of the universal curves \eqref{uncurve} and \eqref{uncurvecpn} to define the {\em phase-tropical limit} of the scaled sequence \eqref{falpha}.
This technique is borrowed from \cite{KaMi}, where it is used also for definition of further tropical notions in the limiting tropical curve, in particular, differential forms.

By a straight holomorphic cylinder in $\ctor$ we mean the subset
\[
Z=\{ b(z_1^{a_1},\dots,z_n^{a_n}) \ | \ z\in\C^\times \}\subset\ctor,
\]
where $a=(a_1,\dots,a_n)\in\Z^n$ and $b\in\ctor$.
This set is the multiplicative translate $bT_a$ of the 1-dimensional torus subgroup
$T_a=\{ (z_1^{a_1},\dots,z_n^{a_n}) \ | \ z\in\C^\times \}$ of $\ctor$.

\begin{lem}\label{lem-punctures}
Let $\phi:A\to\ctor$ be a non-constant 
(parameterized) proper algebraic curve.
Then the number of punctures of $A$ is at least two.
Furthermore, if the number of punctures of $A$ is two then the image of $A$ is a straight holomorphic cylinder in $\ctor$.
\end{lem}
\begin{proof}
The closure $\overline{\phi(A)}\subset\cp^n$ of $\phi(A)$ in $\cp^n\supset\ctor$ must intersect each of the $(n+1)$ coordinate hyperplanes in $\cp^n$. Thus $\overline{\phi(A)}\setminus\ctor$
contains at least two points. By the properness of $\phi$ each of these points must correspond to a puncture.

Suppose that $A$ has two punctures. Denote by $\delta\in H_1(\ctor)=\Z^n$ the class of the image of the simple positive loop going around one of the punctures. Then the similar loop for the other puncture gives is the class $-\delta$.
Let $\pi_\delta:\ctor\to(\C^\times)^{n-1}$ be a surjective homomorphism whose kernel is $T_\delta$.
Then the image of a loop around each puncture of $A$ under $\pi_\delta\circ\phi$ is homologically trivial. Thus $\pi_\delta\circ\phi(A)$ is constant.
\end{proof}

\begin{coro}\label{nosinglepoint}
Suppose that the family \eqref{falpha} converges to a stable curve $\bar f:\overline C\to\cp^n$ when $t_\alpha\to\infty$. Then each spherical component $K\subset\overline C$ must be adjacent at least to two distinct marked or nodal points.
\end{coro}
\begin{proof}
The restriction $\bar f|_K$ is non-constant by the stability of $\bar f$.
Thus $\bar f(K)\cap(\C^\times)^l\neq\emptyset$ for a coordinate subspace (intersection of $n-l$ coordinate hyperplanes) $\cp^l\subset\cp^n$, $1\le l\le n$.
By Lemma \ref{lem-punctures}, ${\bar f}^{-1}(\cp^l\setminus (\C^\times)^l)$ consists of at least two points. The loops going around these points are non-trivial in $H_1((\C^\times)^l)$, and thus they cannot come as limits of boundaries of non-trivial holomorphic disks in $\ctor$.  
\end{proof}

\begin{coro}\label{prescomp}
Let $\widehat C$ be the stable curve obtained as the image by \eqref{ft} of a stable curve $\bar f:\overline C\to\cp^n$ obtained as the limit of the family \eqref{falpha} when $t_\alpha\to\infty$.
If $K\subset\overline C$ is a component which is either non-spherical or is adjacent to at least three nodal or marked points then the map \eqref{uft} maps $K$ isomorphically to a component of $\widehat C$. In other words, the map \eqref{uft} does not contract $K$ to a point.
\end{coro}
\begin{proof}
By Corollary \ref{nosinglepoint} the only components contracted by $\tilde{\operatorname{ft}}|_{\overline C}$ are spherical components with two nodal or marked points. Their contraction does not reduce the number of nodal or marked points adjacent to other components of ${\overline C}$ .
\end{proof}

Suppose that the family $\overline C_\alpha$ converges to $\widehat C\in\bm_{g,k}$ when $t_\alpha\to\infty$.
Let $z_\alpha\in C_\alpha$ be a family of points converging to $z=\lim\limits_{t_\alpha\to\infty}z_\alpha\in \widehat C\subset\bu_{g,k}$.
Let 
\begin{equation}\label{tau}
\tau_\alpha=|f_\alpha(z_\alpha)|^{-1}\in\R^n_{>0}\subset\ctor
\end{equation}
be obtained from $f_\alpha(z_\alpha)$ by taking the absolute inverse value coordinate-wise. 
Then we have $\tau_\alpha f_\alpha(C_\alpha)\cap\Log^{-1}_{t_\alpha}(0)\neq\emptyset$ for the multiplicative translate
$\tau_\alpha f_\alpha(C_\alpha)$ of $f_\alpha(K_\alpha)$ in $\ctor$ by $\tau_\alpha$.
In particular, if the family 
\begin{equation}\label{taufalpha}
\tau_\alpha f_\alpha:C_\alpha\to\ctor
\end{equation}
closes up to a converging family 
\begin{equation}\label{btaufalpha}
\overline{\tau_\alpha f_\alpha}:\overline C_\alpha\to\cp^n
\end{equation}
in $\bm_{g,k}(\cp^n,d)$ then the limit 
\begin{equation}\label{bftau}
\bar f_\tau:\overline C_\tau\to\cp^n
\end{equation}
of the family \eqref{btaufalpha} when $t_\alpha\to\infty$ has a component $K\subset \overline C_\tau$
containing an accumulation point $z_\tau$ of $z_\alpha\in C_\alpha$ inside $\bu_{g,k}(\cp^n,d)$. 
The point $z_\tau$ is mapped to $z$ under the map $\overline C_\tau\to\widehat C$ induced by the map \eqref{uft}.
By Corollary \ref{prescomp}, if $z\in\widehat C$ is not a nodal or marked point then $\tilde{\operatorname{ft}}|_K$ is an isomorphism to its image.
Thus $z_\tau$ is the unique accumulation point, i.e. the limit of the points $z_\alpha$ in $\bu_{g,k}(\cp^n,d)$.
We get the following proposition.
\begin{prop}
Suppose that \eqref{falpha} is a scaled family such that $\overline C_\alpha$ converges to $\widehat C\in\bm_{g,k}$, and $z_\alpha\in C_\alpha$ is a family of points in \eqref{falpha} converging to a point $z\in\widehat C$ in $\bu_{g,k}$, and such that \eqref{btaufalpha} converges to \eqref{bftau}, where $\tau_\alpha\in\R^n_{>0}$ are defined by \eqref{tau}.
If $z\in\widehat C$ is neither nodal nor marked then 
the lift of $z_\alpha$ to $\bu_{g,k}(\cp^n,d)$ under \eqref{uft} with the help of \eqref{btaufalpha} also converges to a point of $\overline C_\tau$.
\end{prop}

\begin{prop}\label{Kzw}
Under the hypotheses of the previous proposition, suppose in addition that $w_\alpha\in C_\alpha$ is another family of points converging to a point $w\in\widehat C$ in $\bu_{g,k}$.
If $z$ and $w$ are non-nodal, non-marked, and belong to the same component $K\subset\widehat C$ then $\overline{\sigma_\alpha f_\alpha}:C_\alpha\to\tor$, $\sigma_\alpha=|f_\alpha(w_\alpha)|^{-1}$, also converges to a stable curve $\bar f_\sigma:\overline C_\sigma\to\cp^n$.

Furthermore, the restrictions $\bar f_\sigma|_{K^\circ_\sigma}$ and $\bar f_\tau|_{K_\tau^\circ}$ both take value in $\ctor$ and one map is obtained from the other by multiplication by $\lim\limits_{t_\alpha\to\infty}\sigma_\alpha/\tau_\alpha$ (which exists).
Here $K_\sigma\subset\overline C_\sigma$ and $K_\tau\subset\overline C_\tau$
are the components corresponding to $K$ under \eqref{uft} and the superscript $K^\circ$ for a component $K$ signifies the complement of the set of nodal and marked points.
\end{prop}
\begin{proof}
The limit $\lim\limits_{t_\alpha\to\infty}\sigma_\alpha/\tau_\alpha$ coincides with $\bar f_\tau(w_\tau)$, where $w_\tau\in K_\tau$ is the point corresponding to $w$ under the isomorphism $K_\tau\approx K$. Multiplication by $\sigma_\alpha/\tau_\alpha\in\R^n_{>0}$ identifies $\tau_\alpha f_\alpha$ into $\sigma_\alpha f_\alpha$.
\end{proof}

\begin{defn}\label{pht-lim}
We say that a scaled family \eqref{falpha} of curves of degree $d$ in $\cp^n\supset\ctor$ converges {\em phase-tropically} if the following conditions hold:
\begin{itemize}
\item
The family $\overline C_\alpha$ converges to a curve $\widehat C\in \bm_{g,k}$, $t_\alpha\to\infty$.
\item 
For each component $K\subset\widehat C$ there exists a point $z\in K^\circ$ and a family $z_\alpha\in C_\alpha$ converging to $z\in\widehat C$ in $\bu_{g,k}$ such that \eqref{btaufalpha} converges to a map 
in $\bm_{g,k}(\cp^n,d)$
 and such that
 \begin{equation}\label{hK}
h_K=\limtrop f_{\alpha}(z_\alpha)\in [-\infty,\infty]^n 
\end{equation}
exists.
Here $K^\circ$ is the complement of the nodal and marked points in $K$.
\end{itemize}

If $h_K\in\R^n$ then the component $K$ is called {\em tropically finite}, otherwise infinite.
The source $\overline C_\tau$ of the limiting map \eqref{bftau} of \eqref{btaufalpha} has a component corresponding to $K$ by Corollary  \ref{prescomp}.
We refer to the map
\begin{equation}\label{Kphase}
\phi_K:K^\circ\to\ctor
\end{equation}
defined by $\phi_K=\bar f_\tau|_{K^\circ}$
as well as its image 
\begin{equation}\label{Kphi}
\Phi(K)=\phi_K(K^\circ)\subset\ctor
\end{equation}
as the {\em phase} of the component $K\subset\widehat C$.
Phases are defined up to multiplication by an element of $\R_{>0}^n$ in $\ctor$.
\end{defn}

The following proposition shows independence of the phases from the choice of $z_\alpha$ and thus justifies the notations free from $z_\alpha$ or $\tau$.

\begin{prop}\label{indepK}
In Definition \ref{pht-lim} 
neither the tropical limit \eqref{hK} nor the phase \eqref{Kphase} of a component $K\subset\widehat C$ depends on the choice of a family $z_\alpha\in C_\alpha$ converging to a point $z\in K^\circ\subset\widehat C$. 
\end{prop}
\begin{proof}
For two choices $z_\alpha\to z$ and $w_\alpha\to w$ with $z,w\in K^\circ$ the limit of $|z_\alpha|/|w_\alpha|$ exists by Proposition \ref{Kzw}. Thus the tropical limits (involving rescaling by $1/\log t_\alpha$) of $z_\alpha$ and $w_\alpha$ must coincide.
The phase of $K$ does not depend on the choice of $z_\alpha$ by Proposition \ref{Kzw}.
\end{proof}

A nodal or marked point $p\in K\subset\widehat C$ corresponds to a puncture of $K^\circ$.
Define $\gamma_p\subset K^\circ$ to be a simple loop going around $p$ in the negative direction with respect to $p$ (and thus in the positive direction with respect to $K\setminus\{p\}$), and set
\[
\delta_K(p)=[\phi_K(\gamma_p)]\in H_1(\ctor)=\Z^n.
\]
Compactifying $\phi_K:K^\circ\to\ctor$ to $\bar\phi_K:K\to\cp^n$ and expanding $\bar\phi$ at $p$ to a series in the corresponding affine coordinates, we get the following proposition.
\begin{prop}[cf. \cite{Mi-am}]
If $\delta_K(p)\neq 0$ then the limit
\[
\Phi(K,p)=\lim\limits_{s\to+\infty}s^{\delta_K(p)}\Phi(K),
\]
$s\in\R_{>0}$, is a straight holomorphic cylinder given by $\{bz^{\delta_K(p)} \ | \ z\in\C^\times\}$ for some $b\in\ctor$. The notations $s^{\delta_K(p)}\in\R_{>0}$ and $z^{\delta_K(p)}\in\ctor$ refer to taking power coordinatewise by $\delta_K(p)$.
If $\delta_K(p)=0$ then $\Phi(K,p)=\Phi(K)$. 
\end{prop}
Recall that $\Phi(K,p)$ as well as $\Phi(K)$ is defined up to a multiplicative translation.
The notations $s^{\delta_K(p)}\in\R_{>0}$ and $z^{\delta_K(p)}\in\ctor$ in the proposition above refer to taking power coordinatewise by $\delta_K(p)$.

\begin{prop}\label{nody}
Suppose that  \eqref{falpha} converges phase-tropically and $p_\alpha\in C_\alpha$ is a family of points converging to a point $p\in K\subset\widehat C$ which is either nodal or marked.
Then for a sufficiently small neighborhood $U\subset\bu_{g,k}$ of $p\in\widehat C\subset\bu_{g,k}$  the limit
\[
\Phi(p)=\lim\limits_{t_\alpha\to\infty} |f_\alpha(p_\alpha)|^{-1}f(U\cap C_\alpha)\subset\ctor
\]
(with respect to the Hausdorff metric on neighborhoods of compacts) exists
and coincides with $\Phi(K,p)$ if $\delta_K(p)\neq 0$.

Furthermore, if $K$ is tropically finite then any accumulation point of $\Log_{t_\alpha}(f_\alpha(p_\alpha))\in\R^n$ is contained in the ray $R_{K,p}\subset\R^n$ emanating from $h_K\in\R^n$ in the direction of $\delta_K(p)$. 
\end{prop} 
\begin{proof}
To prove the convergence it suffices to show that each subsequence has a subsequence convergent to $\Phi(K,p)$. 
Passing to a subsequence, we may assume that $|f_\alpha(p_\alpha)|^{-1}f(C_\alpha)$ yields a convergent family in $\bm_{g,k}(\cp^n,d)$ with the limiting curve $\bar f_p:\overline C_p\to\cp^n$,  $\pi:\overline C_p\to\widetilde C$, such that $z_\alpha$ converges to $z\in\overline C_p$ in $\bu_{g,k}(\cp^n,d)$.
Since $\delta_K(p)\neq 0$, the point $z$ cannot be a nodal or marked point of $\overline C_p$ and must be contained in a component $K_z\subset\overline C_p$ contracted by $\pi$. 
By Lemma \ref{lem-punctures}, $\bar f_p(K_z^\circ)\subset\ctor$ is a straight holomorphic cylinder.
To see that it coincides with $\Phi(K,p)$ it suffices to change the coordinates in $\ctor$ (and accordingly, the toric compactification $\cp^n\supset\ctor$) so that $\delta_K(p)=(0,\dots,0,-n)$, with $n\in\Z_{>0}$.
Then the projectivization $\bar \phi_K:K\to \cp^n$ of \eqref{Kphase} maps $p$ to a point inside the $n$th coordinate hyperplane, and the argument of $\bar\phi_K(p)\in (\C^\times)^{n-1}\times\{0\}$ determines the argument of $\Phi(p)$.

To locate accumulation points of the sequence $\Log_{t_\alpha}(f_\alpha(p_\alpha))$ we compare it against the sequence $\Log_{t_\alpha}(f_\alpha(z_\alpha)$ in \eqref{hK}. The ratio $f_\alpha(p_\alpha)/f_\alpha(z_\alpha)$ converges to $\bar\phi_K(p)\in (\C^\times)^{n-1}\times\{0\}$ and thus the first $(n-1)$ coordinates of $\Log_{t_\alpha}(f_\alpha(p_\alpha)/f_\alpha(z_\alpha))$ go to zero while the $n$th coordinate is essentially non-positive.
\end{proof}



Let \eqref{falpha} be a  phase-tropically convergent family with the limiting curve $\widehat C\in\bm_{g,k}$.
Let $\widehat C^\circ\subset\widehat C$ be the union of tropically finite components of $\widehat C$.
(Note that $\widehat C^\circ$ may be disconnected or empty.)

We define the extended dual subgraph $\tilde\Gamma$ of $\widehat C^\circ$ to incorporate not only its nodal, but also its marked points in the following way.
The {\em vertices} $v_K\in\tilde\Gamma$ correspond to the components $K\subset\widehat C^\circ$.
{\em Bounded edges} $E_p\subset\tilde\Gamma$ correspond to the nodal points $p\in\widehat C^\circ$,
they connect the vertices corresponding to the adjacent tropically finite components (could be the same component).
Vertices together with bounded edges form the dual graph $\Gamma(\widehat C^\circ)$.
To get the open finite graph $\tilde \Gamma$ we attach to $\Gamma(\widehat C)$ the {\em leaves}, or half-infinite edges, $E_q\approx [0,+\infty)$ corresponding to marked points $q\in\widehat C$ contained in tropically finite components $K$, and also to nodal points $q\in\widehat C$ adjacent simultaneously to a tropically finite component $K$ and to an infinite component of $\widehat C$.
We attach $E_q\approx [0,+\infty)$ to $\Gamma(\widehat C^\circ)$ by
identifying $0\in [0,+\infty)$ with $v_K\in \Gamma(\widehat C^\circ)$. 

Our next goal is to define
\[
\tilde h:\tilde\Gamma\to\R^n.
\]
We set $\tilde h(v_K)=h_K\in\R^n$ using \eqref{hK}.
Let $p\in \widehat C$ be a nodal point between components $K$ and $K'$.
If $v_K=v_{K'}$ we define $\tilde h|_{E_p}$ to be the constant map to $h_K$.
If $v_K\neq v_{K'}$ then y Proposition \ref{nody} $\tilde h(K')-\tilde h(K)=s\delta_p(K)$ with $s> 0$.
In particular, in this case $\delta_K(p)=\delta_{K'}(p)\neq 0$.
We identify a bounded edge $E_p$ with the Euclidean interval of length $s$, and define $\tilde h|_{E_p}$ to be the affine map to the interval $[\tilde h(K),\tilde h(K')]$.
We identify a leaf $E_q$ with the Euclidean ray $[0,+\infty)$, and define $\tilde h|_{E_q}$
to be the affine map 
to the ray emanating from $\tilde h(K)$ in the direction of $\delta_q(K)$ stretching the length $|\delta_q(K)|$ times.

Define $\Gamma$ to be the (open finite) graph obtained from $\tilde \Gamma$ by contracting all edges collapsed to points by $\tilde h$ and
\begin{equation}\label{map-h}
h:\Gamma\to\R^n
\end{equation}
to be the map induced by $\tilde h$.
Each vertex $v\in\Gamma$ corresponds to a connected subgraph $\tilde\Gamma_v\subset\tilde\Gamma$.
We define the genus function $g(v)$ to be the sum of the genera of all components of $\widetilde C$ corresponding to the vertices of $\tilde\Gamma_v$ and the number of cycles in $\tilde\Gamma_v$.
\begin{prop}\label{prop-vancircle}
The map \eqref{map-h} is a parameterised tropical curve.
\end{prop}
\begin{proof}
Each edge $E_p\subset\Gamma$ corresponds to a marked or nodal point $p$ of $\widehat C$ and thus to an embedded {\em vanishing} circle $\gamma_p$ of $C_\alpha$ for large $t_\alpha$.
The circle $\gamma_p$ is oriented by the choice of a component $K\subset\widehat C$ containing $p$, and thus by a choice of the vertex $v_K$ adjacent to $E_p$.
This choice is equivalent to the orientation of the $E_p$, and thus to the choice of the unit tangent vector to $E_p$. The image $u(E_p)\in\Z^n$ of this vector under $dh$ is given by $\delta_K(p)$
The balancing condition \eqref{balancing} follows from the homology dependance of $\gamma_p$ given by $K^\circ$.
\end{proof}

\begin{defn}
The {\em phase-tropical limit} of a phase-tropically converging family \eqref{falpha} consists of the parameterised tropical curve \eqref{map-h} as well as the phases \eqref{Kphi} for the components $K\subset\widehat C$.
\end{defn}

\begin{rem}
Consideration of $\Gamma$ instead of $\tilde\Gamma$ allows us to ignore tropical lengths of the edges of $\tilde\Gamma$ collapsed by $\tilde h$. It is possible to define the limiting tropical length also on these edges. The resulting edge might appear to be not only finite, but also zero or infinite, see \cite{KaMi}.
\end{rem}

\newcommand{\Arg}{\operatorname{Arg}}
For the following definition we use the identification 
$$\ctor=\R^n\times (S^1)^n$$
given by the (logarithm) polar coordinates identification of $z\in\ctor$ with $(\Log(z),\Arg(z))$,
where $\Arg$ refers to the map of taking the argument coordinatewise. 
The closure of the image 
\[
\Arg_K=\overline{\Arg(\Phi(K))}\subset (S^1)^n.
\] 
is known as the closed {\em coamoeba} of $\Phi(K)\subset\ctor$.
Since $\Phi(p)\subset\ctor$ is a straight holomorphic cylinder, the image $\Arg_p=\Arg(\Phi(p))$ is a geodesic circle in the flat torus $(S^1)^n$.
\begin{defn}
The {\em unparameterised phase-tropical limit} of a phase-tropically converging family \eqref{falpha}
is the set 
\begin{equation}\label{l-psi}
\Psi=
\bigcup\limits_{v_K}
\{\tilde h(v_K)\}\times\Arg_K
\cup
\bigcup\limits_{p}
\{\tilde h(E_p)\}\times\Arg_p,
\end{equation}
where $K$ runs over all components of $\widehat C$ while $p$ runs over all nodal and marked points of $\widehat C^\circ$.
Note that 
$Y=\Log(\Psi)$ 
is an unparameterised tropical curve with the weight data coming from \eqref{map-h}.
\end{defn}

\begin{rem}
Loci \eqref{l-psi} are {\em complex tropical curves} in the terminology of \cite{Mi05}.
In more modern terminology, {\em complex tropical} are replaced with {\em phase-tropical}.
\end{rem}

\begin{theorem}\label{phtr-thm}
If a scaled family of holomorphic curves $f_\alpha:C_\alpha\to\ctor$ converges phase-tropically then for any family $z_\alpha\in C_\alpha$ such that $\limtrop f_\alpha(z_\alpha)\in\R^n$ and $\lim\limits_{t_\alpha\to\infty}\Arg(z_\alpha)\in (S^1)^n$ exist we have 
\begin{equation}\label{limPsi}
(\limtrop f_\alpha(z_\alpha),\lim\limits_{t_\alpha\to\infty}\Arg(z_\alpha))\in\Psi.
\end{equation}
Conversely, any point of $\Psi$ can be presented in the form \eqref{limPsi} for some family $z_\alpha\in C_\alpha$.
\end{theorem}
\begin{proof}
Passing to a subfamily if needed, we may assume that $z_\alpha$ converge to a point $z\in\widehat C^\circ$.
If $z\in K^\circ$ for some component $K\subset\widehat C^\circ$ then by Proposition \ref{indepK} $\limtrop f_\alpha(z_\alpha)=\tilde h(v_K)$, while $\lim\limits_{t_\alpha\to\infty}\Arg(z_\alpha)=\Arg(\phi_K(z))$.
Conversely, to present a point $(\tilde h(v_K),\Arg(\phi_K(z)))$ in the form \eqref{limPsi}, it suffices to approximate $z$ by $z_\alpha\in C_\alpha$.
To present a point $(\tilde h(v_K),\xi)$, $\xi\in\Arg_K\setminus\Arg(\Phi(K))$, we first approximate $\xi$ with points from $\Arg(\Phi(K))$, approximate them as above, and then use the diagonal process.

If $z\in\widehat C^\circ$ is a nodal or marked point then by Proposition \ref{nody}
$(\limtrop f_\alpha(z_\alpha),\lim\limits_{t_\alpha\to\infty}\Arg(z_\alpha))\in\{\tilde h(E_p)\}\times\Arg(\Phi(p))$.
Consider a small neighborhood $W\ni p$ in the universal curve $\bu_{g,k}$ such that $W\cap\widehat C$ consists of one or two disks (depending on whether $p$ corresponds to a leaf or to a bounded edge of $\tilde\Gamma$).
Then the image $(\Log_{t_\alpha},\Arg) (f_\alpha (W\cap C_\alpha))$ is a connected annulus converging to $\{\tilde h(E_p)\}\times\Arg_p$. This implies that any point in $\{\tilde h(E_p)\}\times\Arg_p$ is presentable in the form \eqref{limPsi}.
\end{proof}

\begin{theorem}[Phase-tropical limit compactness theorem]
Let $f_\alpha:C_\alpha\to\tor\subset\cp^n$, $t_\alpha\to+\infty$, $\alpha\in A$, be a scaled family of curves of degree $d$,
where the source curves $C_\alpha$ are Riemann surfaces of genus $g$ with $k$ punctures.
Then there exists a scaling subsequence $t_\beta\to+\infty$, $\beta\in B\subset A$, such that the subfamily $f_\beta:C_\beta\to\tor$ converges phase-tropically.
\end{theorem}
\begin{proof}
By compactness of $\bm_{g,k}$ we may ensure convergence of $\overline C_\alpha$ to $\widehat C\in\bm_{g,k}$ after passing to a subfamily.
For each component $K\subset\widehat C$ we choose $z\in K^\circ$ and a small transversal disk $\Delta_z\subset\bu_{g,k}$ to $\widehat C$ at $z$.
Then for large $t_\alpha$ the intersection $\Delta_z\cap C_\alpha\subset\bu_{g,k}$
consists of a single point, and defines a family $z_\alpha\in C_\alpha$ converging to $z$.
Compactness of 
$\bm_{g,k}(\cp^n,d)$ and that of $\R\cup\{\pm\infty\}$ ensures convergence of the family \eqref{btaufalpha} and the limit \eqref{hK} after passing to a subfamily.
\end{proof}

\subsection{Tropical limits of non-commutative amoebas in $\Hy$}
\newcommand{\ok}{\overline\varkappa}
It turns out that non-commutative amoebas, considered in the first three sections of the paper, also admit interesting tropical limits.
But, due to non-commutativity of hyperbolic translations, passing to such limit is only possible once we distinguish the origin point $0\in\Hy$.
This choice determines the rescaling map
$
\Hy\ni x\mapsto sx\in \Hy
$ 
for every $s>0$ homeomorphically mapping $\Hy$ to itself.
This map extends to a homeomorphism 
\[
\bHy\stackrel{\approx}\to\bHy,\ x\mapsto sx,
\]
by setting it to be the identity on $\dd\Hy$.
We define 
\begin{equation}\label{varkappat} 
\varkappa_t:\I\to\Hy, \ \varkappa_t(z)=\frac1{\log t}\varkappa(z),
\end{equation}
$t>1$, as the rescaling of the map \eqref{varkappa},
and denote by
\begin{equation}\label{barvarkappat} 
\overline\varkappa_t:\overline\I\to\bHy, 
\end{equation}
the corresponding compactified map.
The images $\ok(V)\subset\bHy$ of algebraic varieties $V\subset\overline\I$ are nothing but their rescaled hyperbolic amoebas.
In particular, they are closed sets.

In this section we show that when $t\to\infty$ these rescaled amoebas of curves converge to the {\em $\Hy$-tropical spherical complexes} we define below.
\begin{defn}
The {\em $\Hy$-floor diagram} $\Delta$ of degree $d>0$ is a 
finite graph with the set of vertices $\operatorname{Vert}(\Delta)$, the set of edges $\operatorname{Edge}(\Delta)$, and the following additional data.
 \begin{itemize}
\item There is a map $$r:\operatorname{Vert}(\Delta)\to [0,\infty]$$ called the vertex {\em width}.
We distinguish the subset $$\operatorname{Vert}^0(\Delta)=r^{-1}(0), \
\operatorname{Vert}^{+}(\Delta)=r^{-1}(0,\infty),\
\operatorname{Vert}^{\infty}(\Delta)=r^{-1}(\infty)
$$ of vertices of zero, positive and infinite width.
\item There is a map $$\varphi:\operatorname{Edge}(\Delta)\to \dd\Hy=S^2$$ called the edge {\em angle}.
\item There is a map $$(d_+,d_-):\operatorname{Vert}^{+}(\Delta)
\cup\operatorname{Vert}^{\infty}(\Delta)
\to\Z_{\ge 0}^2$$ called the vertex {\em bidegree} as well as a {\em degree} map
$$\delta:\operatorname{Vert}^{0}(\Delta)\to\Z_{\ge 0}.$$
\item There is an {\em edge weight} map $$w:\operatorname{Edge}(\Delta)\to\Z_{>0}.$$
\end{itemize}
This data is subject to the following properties.
\begin{itemize}
\item
No edge may connect vertices of the same width. 
In particular, $\Delta$ is loop-free.
\item
\begin{equation}
\label{total-deg}
\sum\limits_{v\in\operatorname{Vert}^0(\Delta)}\delta(v)+\sum\limits_{v\in\operatorname{Vert}^+(\Delta)}(d_+(v)+d_-(v))=d.
\end{equation}

\item For every $v\in \operatorname{Vert}(\Delta)$ we define $\operatorname{div}(v)$ to be the sum of the weights of the edges connecting $v$ to vertices whose width is larger than $v$ minus the sum of the weights of the edges connecting $v$ to vertices whose width is smaller than $v$ and require that
\begin{equation}\label{eq-div}
2(d_+(v)+d_-(v))=\operatorname{div}(v), 
 \ \ \text{and}\ \
2\delta(v_0)=\operatorname{div}(v_0)
\end{equation}
for $v\in\operatorname{Vert}^+(\Delta)$ and $v_0\in\operatorname{Vert}^0(\Delta)$.

\item
If $d_+(v)=0$, $v\in \operatorname{Vert}^+(\Delta)\cup\operatorname{Vert}(\Delta)^{\infty}$, then we have
\begin{equation}
\label{d+0}
\varphi(E)=\varphi(E')
\end{equation}
whenever $E,E'$ are two edges of $\Delta$ adjacent to $v$.
In this case we set $\varphi(v)=\varphi(E)$.

\end{itemize}
\end{defn}
Given $0\le r\le \infty$ and $\varphi\in \dd\Hy=S^2$ we define $(r,\varphi)\in\bHy$ to be the point on the compactified geodesic ray $R_\varphi\subset\bHy$ connecting $0\in\Hy$ to $\phi$ such that the distance between $(r,\varphi)$ and $0$ equals to $r$.
For $0\le r_1\neq r_2\le \infty$ we define $R_\varphi[r_1,r_2]\subset R_\varphi$ to be the interval of points whose distance to $0$ is between $r_1$ and $r_2$.
Denote by $S^2(r)\subset\bHy$ the sphere of radius $0\le r\le\infty$ and center $0$, in particular, $S^2(0)=\{0\}$, $S^2(\infty)=\dd\Hy$.
The coordinates $(\rho,\varphi)$ can be thought of as polar coordinates in $\bHy$.
The first coordinate gives the map
\begin{equation}\label{map-rho}
\rho:\bHy\to [0,\infty]
\end{equation}
measuring the distance to the origin $0\in\Hy$.
The second coordinate gives the map
\begin{equation}\label{map-phi}
\varphi:\bHy\setminus \{0\}\to \dd\Hy=S^2
\end{equation}
corresponding to the projection from the origin $0$ to the absolute $\dd\Hy$.

Let $v\in\operatorname{Vert}(\Delta)$.
If $d_+(v)>0$ we define $\Theta(v)=S^2(r(v))$.
If $d_+(v)=0$ we define $\Theta(v)=\{(r(v),\varphi(v))\}$.
For $E\in\operatorname{Edge}(\Delta)$ an edge connecting vertices $v_1$ and $v_2$ we define $\Theta(E)=R_{\varphi(E)}[r(v_1),r(v_2)]$.

\begin{defn}
The {\em $\Hy$-tropical spherical complex} associated to an { $\Hy$-floor diagram} $\Delta$ is
the set
\begin{equation}\label{Theta}
\Theta(\Delta)=\bigcup\limits_{v\in\operatorname{Vert}(\Delta)} \Theta(v)\ \ \cup\bigcup\limits_{E\in\operatorname{Edge}(\Delta)}\Theta(E)\ \subset \bHy.
\end{equation}
\end{defn}

\begin{exa}
Figure \ref{exa-tsc} depicts a $\Hy$-tropical spherical complex of degree 3. It consists of two concentric circles (representing spheres in the actual 3D picture) corresponding to two vertices of $\Delta$ of bidegree $(1,0)$. One edge has an endpoint inside the inner circle. This endpoint correspond to a vertex of bidegree $(0,1)$ and thus gives a point rather than a sphere in $\Hy$. This vertex has a single adjacent edge of weight 2 and thus conforms to \eqref{eq-div} . All other edges have weight 1.
The six outermost edges end at the absolute at six vertices of bidegree $(0,0)$.
\begin{figure}[h]
\includegraphics[height=98mm]{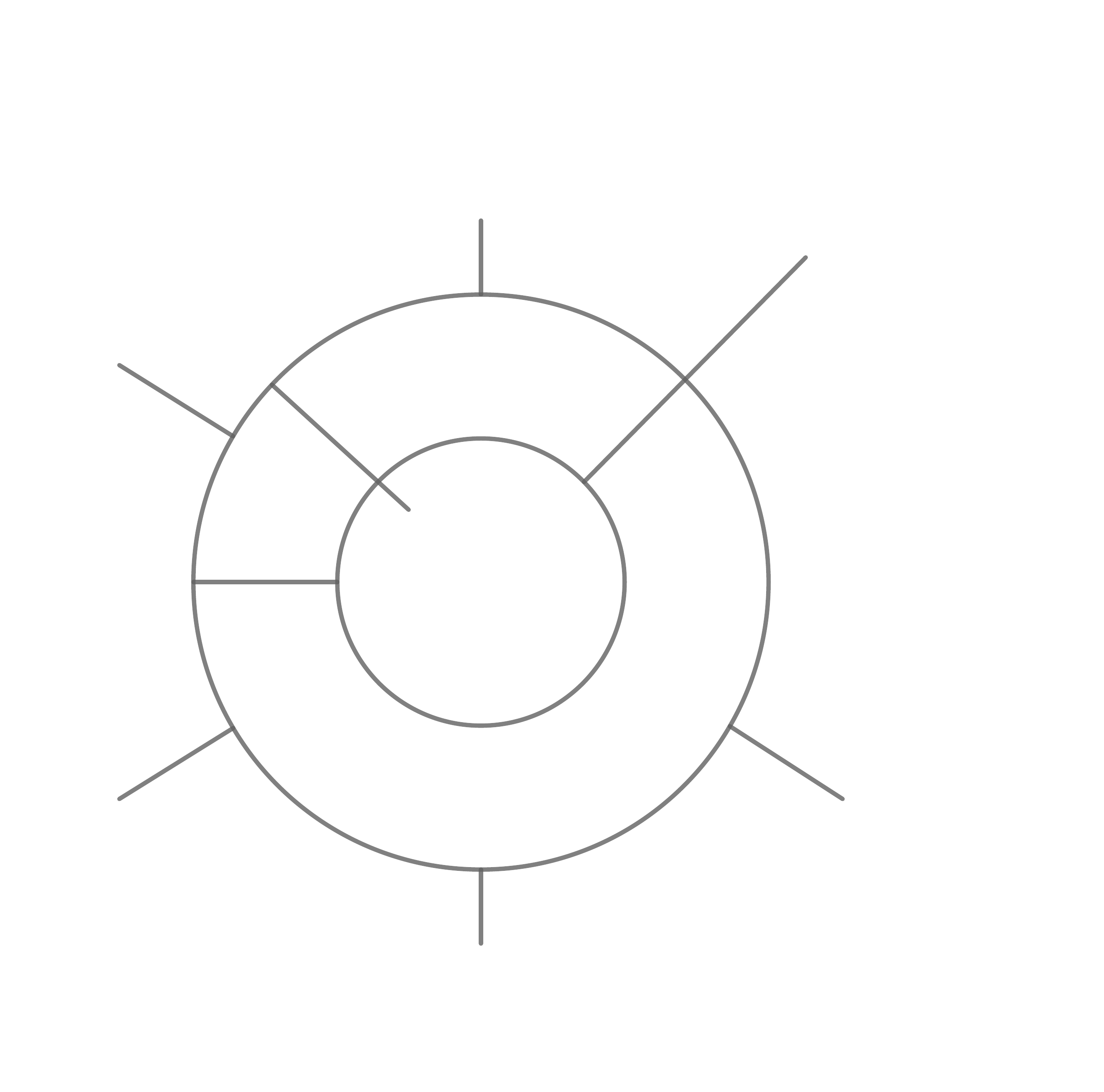}
\caption{A $\Hy$-tropical spherical complex is a limit of amoebas of curves.
\label{exa-tsc}}
\end{figure}
\end{exa}


Consider the locus
\[
P=\varkappa^{-1}(0)=\operatorname{SO}(3)\approx\rp^3\subset\I.
\]
It is the fixed locus of the antiholomorphic involution 
\begin{equation}\label{Pconj}
\begin{pmatrix}
a & b\\
c & d
\end{pmatrix}
=A \mapsto \bar A^*=
\begin{pmatrix}
\bar d & -\bar b\\
-\bar c & \bar a
\end{pmatrix}
\end{equation}
In homogeneous coordinate $(a+d:i(a-d):ib:ic)$ this involution is the standard complex conjugation, so $P\approx\rp^3$ can be thought of the real locus of $\overline\I=\cp^3$.
In particular, after this coordinate change, the lines in $\cp^3$ invariant with respect to the involution \eqref{Pconj} are nothing but the lines defined over $\R$. We call such lines {\em $P$-real} lines.
Their intersection with $P$ is diffeomorphic to the circle $\rp^1$.
\begin{lem}\label{geo0}
A line $l\subset\overline\I=\cp^3$ is $P$-real if and only if its amoeba $\varkappa(l\cap\I)$ is a geodesic line passing through $0\in\Hy$.
\end{lem}
\begin{proof}
Suppose that $\varkappa(l\cap\I)$ is a geodesic line passing through $0\in\Hy$.
Then, by Proposition \ref{line-descr}, $\varkappa^{-1}(0)\cap l\approx S^1$.
But $l$ is real if and only if $l\cap\rp^3$ is infinite.

Conversely, if a line $l$ is $P$-real then by Proposition \ref{line-descr} its amoeba is a geodesic containing the origin.
\end{proof}

For $A\in\cp^3\setminus P$ there exists a single $P$-real line $l_A$ passing through $A$. It is the line passing through $A$ and $\bar A^*$.
Define the map
\begin{equation} \label{piP}
\pi_P:\overline\I\setminus P\to Q
\end{equation}
by setting $\pi_P(A)$ to be one of the two points of intersection $l_A\cap Q$.
Namely, we note that $l\setminus P$ consists of two open half-spheres in the sphere $l\approx\cp^1$, and set $\pi_P(A)$ to be the unique point in $l_A\cap Q$ contained in the same component of $l\setminus l\cap P$ as $A$.

\begin{coro}
The map \eqref{piP} is a continuous map that agrees with the map \eqref{map-phi} under $\overline\varkappa$, i.e. $\overline\varkappa\circ\pi_P=\varphi\circ\overline\varkappa$.
\end{coro}
\begin{proof}
By Lemma \ref{geo0}, the amoeba of the fiber of  \eqref{piP} is the fiber of \eqref{map-phi}.
\end{proof}

\begin{rem}
Clearly, the map \eqref{piP} is not only continuous, but also smooth in the sense of (real) differential topology. In particular, it presents the space $\overline\I$ as a tubular neighborhood of the hyperboloid $Q$ in $\cp^3$. 
All fibers of \eqref{piP} are open hemispheres in some lines in $\cp^3$.
In particular, they are holomorphic curves.
Nevertheless, the map \eqref{piP} is not holomorphic.

To see this, consider a line $l$ close to a generator $\{z\}\times\cp^1$ of the hyperboloid $Q=\cp^1\times\cp^1$, but intersecting $Q$ at two distinct points.
Its image $\pi_P(l)\subset Q$ must be homologous to $\{z\}\times\cp^1$.
If it were holomorphic then $\pi_P(l)$ would have to be a generator itself.
But then we get a contradiction with the inclusion $Q\cap l\subset\pi_P(l)$.
\end{rem}

Suppose that $V_\alpha\subset\cp^3$, $t_\alpha\to\infty$, $\alpha\in A$, is a scaled family of irreducible algebraic curves and $\Theta(\Delta)$ is the $\Hy$-tropical spherical complex associated to an $\Hy$-floor diagram $\Delta$.
For an interval $I=[r_1,r_2]\subset [0,\infty]$ we define
$$
\Theta(I)=\Theta(\Delta)\cap\rho^{-1}(I)\subset\bHy,
$$
and $$\Delta(I)\subset\Delta$$ to be the open subgraph consisting of all vertices $v\in\Delta$ such that $\Theta(v)\subset\Theta(I)$ as well as all open (i.e. not including adjacent vertices) edges $E\subset\Delta$ such that $\Theta(E)\cap \Theta(I)\neq\emptyset.$
Given a component $K_\Delta\subset\Delta(I)$ we define
\[
\Theta_I(K_\Delta)=\Theta(I)\cap (\bigcup\limits_{v}\Theta(v)\cup \bigcup\limits_{E}\Theta(E)),
\]
where $v$(resp. $E$) goes over all vertices (resp. edges) of $\Delta$ contained in $\Delta(I)$.

We define
$
V_\alpha(I)
$
to be the normalization of 
$$
(\rho\circ\overline\varkappa_{t_\alpha})^{-1}(I)\cap V_\alpha,
$$
i.e. the proper transform of $(\rho\circ\overline\varkappa_{t_\alpha})^{-1}(I)$ under the normalization map $\tilde V\to V$.

For $t>0$ define the homeomorphism 
\begin{equation}
H_t:\bHy\to\bHy, 
\end{equation}
by the properties
\begin{itemize}
\item
$\overline\varkappa(H_t(z))=(\frac 1{\log t}\rho,\varphi)$ if $H_t(z)=(\rho,\varphi)$, $z\in \bHy$,
\item
$H_t(z)=z$ if $z\in Q$,
\item
$\iota(H_t(z))=\iota(z)$ if $z\in \bHy\setminus Q$.
\end{itemize}
Here $\iota:\I\to P=\operatorname{SO}(3)\approx\rp^3$ is the coamoeba map \eqref{iota}.
We have
\begin{equation}\label{htpiP}
\overline\varkappa\circ H_t=\overline\varkappa_t.
\end{equation}

\begin{defn}
Let $V_\alpha\subset\cp^3$, $t_\alpha\to\infty$, $\alpha\in A$,
be a scaled family of irreducible algebraic curves and
$\Theta(\Delta)$ be the $\Hy$-tropical spherical complex associated to an $\Hy$-floor diagram $\Delta$.
We say that the family $V_\alpha$ {\em $\varkappa$-tropically converges} to $\Theta(\Delta)$ if 
$$
\overline\varkappa_{t_\alpha} (V_\alpha) \to \Theta(\Delta)\subset\bHy  
$$
when $t_\alpha\to\infty$ (in the Hausdorff metric on subsets of $\bHy$),
and for every interval $[r_1,r_2]=I\subset [0,\infty]$ such that $\dd I\cap r(\operatorname{Vert}^+(\Delta))=\emptyset$
there is a 1-1 correspondence between the components $K_\Delta\subset\Delta(I)$ and the components
$K_\alpha\subset V_\alpha(I)$ for sufficiently large $t_\alpha$ with the following properties.
\begin{itemize}
\item
The amoebas $\overline\varkappa_t(K_\alpha)$ of the subsurfaces $K_\alpha\subset V_\alpha$ converge to $\Theta_I(K_\Delta)$.

\item
If $K_\Delta$ is vertex-free, i.e. consists of a single (open) edge $E\subset\Delta$, then the following conditions hold.
\begin{itemize}
\item
The subsurface $K_\alpha$ is homeomorphic to an annulus.
\item
There exists a point $q(E)\in Q$ such that $K_\alpha$ converges to $\{q\}$, $t_\alpha\to\infty$, in the Hausdorff metric on subsets of $\overline\I$.
The point $q(E)$ depends only on the edge $E$ and not on the interval $I$ (as long as $E\subset\Delta(I)$ is a component so that $q(E)$ is defined).
\item
The annuli $\Psi(K_\alpha)=H_{t_\alpha}(K_\alpha)$ converge to the annulus $$\Psi_I(K_\Delta)=\pi_P^{-1}(q)\cap\varkappa^{-1}(\Theta_I(K_\Delta))$$ in the Hausdorff metric on subsets of $\overline\I$ while going $w(E)$ times around it, i.e. so that in a small regular neighborhood $W\supset \Psi_I(K_\Delta)$ the cokernel of the homomorphism
\[
\Z\approx H_1(\Psi(K_\alpha))\to H_1(W)=H_1(\Psi_I(K_\Delta)) \approx\Z
\]
induced by the inclusion $\Psi(K_\alpha)\subset W$ for large $t_\alpha$ is of cardinality $w(E)$.
\end{itemize}
\item 
If $K_\Delta$ contains a single vertex $v$
then each component of the boundary $\dd K_\alpha$
corresponds to an edge $E$ adjacent to $v$.
Define $U(v)$ to be the union of small ball neighborhoods of the points $q(E)$.
\begin{itemize}
\item
If $v\in\operatorname{Vert}^+(\Delta)\cup\operatorname{Vert}^\infty(\Delta)$
then the fundamental class $[K_\alpha]$ is an element of
$$H_2(\overline\I\setminus P,U(v))=H_2(\overline\I\setminus P)=H_2(Q)=\Z^2$$
and we require that $[K_\alpha]=(d_+(v),d_-(v))$ for sufficiently large $t_\alpha$.
\item
If $v\in\operatorname{Vert}^0(\Delta)$
then the fundamental class $[K_\alpha]$ is an element of
$$H_2(\overline\I,U(v))=H_2(\overline\I)=H_2(\cp^3)=\Z$$
and we require that $[K_\alpha]=\delta(v)$ for sufficiently large $t_\alpha$.
\end{itemize}
\end{itemize}
\end{defn}

\ignore{
Note that the definition implies that the point $q(E)\in Q$ depends only on the edge $E\subset\Delta$, and not on an interval $I\subset (0,\infty)$ such that $\rho(\Theta(E))\cap I\neq\emptyset$
since $E$ is connected.

For a closed interval $I\subset (0,\infty)$ and large $t_\alpha$ we define 
\begin{equation}\label{PsiE}
\Theta_\alpha=\frac1{\log t_\alpha} 
\Psi_\alpha(E)=H_{t_\alpha}(K_\alpha)
\end{equation}
}



\begin{theorem}
Let $V_\alpha\subset\overline\I=\cp^3$, $t_\alpha\to\infty$, $\alpha\in A$, be a scaled family of reduced irreducible algebraic curves of degree $d$ not contained in the quadric $Q\subset\overline\I$.
Then there exists a subfamily $t_\beta\to\infty$, $\beta\in B\subset A$,
such that $V_\beta$ $\varkappa$-tropically converge to
the $\Hy$-tropical spherical complex 
\begin{equation}\label{lim-thm}
\Theta(\Delta)=\lim\limits_{t_\beta\to\infty}\overline\varkappa_{t_\beta} (V_\beta)\subset\bHy
\end{equation}
associated to an $\Hy$-floor diagram $\Delta$.
\end{theorem}
The limit in \eqref{lim-thm} is taken in the Hausdorff metric on subsets of $\bHy$.

\begin{proof}
After passing to a subsequence, we may assume that $V_\alpha\subset\cp^3$ converge to a (possibly reducible or non-reduced) curve $V_\infty\subset\cp^3$ of degree $d$. 
Choose a point $p\in Q\setminus V_\infty$ so that the two lines passing through $p$ and contained in $Q$ are disjoint from the components of $V_\infty$ not contained in $Q$. 
Then the projection from $p$ defines a holomorphic map
\[
\pi_\alpha:V_\alpha\to Q
\]
such that $\pi_\alpha(z)=z$ if $z\in Q\cap V_\alpha$.
Using \eqref{Qproj} we define the holomorphic map
\[
\varphi_\alpha=\pi_+\circ\pi_\alpha:V_\alpha\to \cp^1.
\]
Clearly, if $z_\alpha\in V_\alpha$ is a family such that $\lim\limits_{t_\alpha\to\infty} z_\alpha=z_\infty\in Q$ then
we have $$\overline\varkappa(z_\infty)=(\infty,\lim \limits_{t_\alpha\to\infty} \varphi_\alpha(z_\alpha))\in \bHy$$
in the polar coordinates presentation.

\newcommand{\ctortri}{({\mathbb C}^\times)^3}
Our next objective is to compute the tropical limit of the $\rho$-coordinate
of $\overline\varkappa(z_\alpha)\in\bHy$.
Recall (see \cite[Proposition 2.4.5]{Thu}) that 
\begin{equation}\label{H-hyp}
\Hy=\{(x_0,x_1,x_2,x_3)\in\R^4\ |\ x_0x_3-x_1^2-x_2^2=1\}
\end{equation}
with the metric $$d_{\Hy}((x_0,x_1,x_2,x_3),(y_0,y_1,y_2,y_3))=
\operatorname{arccosh}(\frac{x_0y_3+x_3y_0}2-x_1y_1-x_2y_2).$$
Furthermore, see \cite[Chapter 2.6]{Thu}), the points of $\Hy$ may be identified with the unitary Hermitian matrices
$
\begin{pmatrix}
x_0 & x_1+ix_2\\
x_1-ix_2 & x_3
\end{pmatrix},
$ while the absolute $\dd\Hy$ can be identified with the rank 1 Hermitian matrices up to scalar multiplication.
Any Hermitian $2\times 2$-matrix up to real multiplication is uniquely represented by a unitary Hermitian matrix while the origin $0\in\Hy$ can be identified with the unit matrix.
Thus we have the following presentation of  the map $\overline\varkappa$ for a matrix $A\in\cp^3$ (viewed up to a complex scalar multiplication):
\begin{equation}
\overline\varkappa(A)=AA^*\in\bHy.
\end{equation}
In particular, 
\begin{equation}\label{rhokapt}
\rho(\overline\varkappa_t(A))=\frac{1}{\log t}d_{\Hy}(AA^*,0)=
\frac{1}{\log t}\operatorname{arccosh}\frac{||A||^2}{2|\det(A)|},
\end{equation}
where $||A||^2=|z_0|^2+|z_1|^2+|z_2|^2+|z_3|^2$ for  $A=
\begin{pmatrix}
z_0 & z_1 \\
z_2 & z_3
\end{pmatrix}$.
Note that for $x\le 0$ we have $e^x\le 2\cosh(x)\le e^x+1$ and thus for $t\le 1$ we have 
$\log(2t-1)\le\operatorname{arccosh}(t)\le \log(2t)$.
This implies that for any scaled sequence $y_\alpha\in\R_{\ge 0}$ we have 
$$\limtrop y_\alpha=\lim\limits_{t_\alpha\to\infty}\frac{1}{\log t_\alpha}\operatorname{arccosh}(y_\alpha),$$
so that the limiting value of \eqref{rhokapt} is given by the tropical limit of the quantity
\begin{equation}\label{rasst}
\frac{|z_0|^2+|z_1|^2+|z_2|^2+|z_3|^2}{|z_0z_3-z_1z_2|}=
\sum\limits_{j=0}^3|\frac{z_j^2}{z_0z_3-z_1z_2}|\ge 2,
\end{equation}
where each individual term $\frac{z_j^2}{z_0z_3-z_1z_2}$ is a rational function in the coordinates $z_0,z_1,z_2,z_3$.

Let  $\nu_\alpha:\overline C_\alpha\to V_\alpha\in\bar I=\cp^3$ be the normalization of $V_\alpha$.
Define $\bar f^{Q}_\alpha:\overline C_\alpha\to(\cp^1)^4$ as the composition of $\nu$ and the map
$$
(\frac{z_0^2}{z_0z_3-z_1z_2} , \frac{z_1^2}{z_0z_3-z_1z_2} , \frac{z_2^2}{z_0z_3-z_1z_2} , \frac{z_3^2}{z_0z_3-z_1z_2})\in(\cp^1)^4.
$$
Denote $\bar f^{\dd\Hy}_\alpha=\pi_q\circ\nu_\alpha$ and
$$\bar f_\alpha=(\bar f^{Q}_\alpha,\bar f^{\dd\Hy}_\alpha,\nu_\alpha):\overline C_\alpha\to(\cp^1)^5\times\cp^3.$$ 
Set $C_\alpha=\overline f_\alpha^{-1}((\C^\times)^8)$.
The surface $C_\alpha$ is obtained from $\overline C_\alpha$ by removing $k$ points.
After passing to a subsequence, the number $k$ as well as the genus $g$ of $\overline C_\alpha$ may be assumed to be independent of $\alpha$.
The map 
$$\rho_\alpha:C_\alpha\to [2,\infty]$$
defines as the composition of $\overline f_\alpha$ and \eqref{rasst} is
a real algebraic function on on $\overline C_\alpha$ whose degree is bounded.
Thus, after yet another passing to a subfamily, our curves $C_\alpha$ may be assumed to have a finite number $k_c$ of such critical points, so that we can mark these points on $\overline C_\alpha$.
Passing to a phase-tropically converging subfamily
provided by Theorem \ref{phtr-thm} we get the limiting curve $\widehat C\in\bm_{g,k+k_c}$
with the tropical limits \eqref{hK} and phases \eqref{Kphase} for each component $K\subset\widehat C$.
We also consider the limiting curve $f^{\dd\Hy}_\infty:\widehat C_{\dd\Hy}\to\cp^1$ in $\bm_{g,k+k_c}(\cp^1)$ as well as the forgetting map $\widehat C_{\dd\Hy}\to\widehat C$ contracting some components of $\widehat C_{\dd\Hy}$.
In addition, we consider the limiting map $\nu_\infty\in\bm_{g,k+k_c}(\cp^3,d)$ of $\nu_\alpha$ and the limiting map $f^{\dd\Hy}_\infty\in\bm_{g,k+k_c}(\cp^1,d)$ of $f^{\dd\Hy}_\alpha$.

We build the limiting $\Hy$-floor diagram $\Delta$ with the help of $\widehat C_{\dd\Hy}$.
As in the proof of Theorem \ref{phtr-thm} we first build a finer graph $\tilde\Delta\to\Delta$.
Namely, we define a vertex of $\tilde\Delta$ as either a component $K\subset \widehat C_{\dd\Hy}$ or a marked point $p\in \widehat C_{\dd\Hy}$ 
corresponding to an intersection point of $\bar f_\alpha(\overline C_\alpha)\cap Q$ for large $t_\alpha$.
We set $r(v_p)=\infty$ and $d_+(v_p)=d_-(v_p)=0$ if $v_p\in\operatorname{Vert}(\tilde\Delta)$ comes from such a marked point $p$.

Since the tropical limit of the sum equals to the maximum of the tropical limits of the summands, the limit
$$
\rho_K=\limtrop \rho_\alpha(z_\alpha)\in [0,\infty],
$$
where $z_\alpha\in \overline C_\alpha\subset \bu_{g,k+k_c}(\cp^1)$ is a family converging to a non-nodal non-marked point of $K$
depends only on $K$, 
and must be equal to $\infty$ for any non-tropically finite component $K$.
Namely, $\rho_K$ coincides with the maximum of the first four coordinates of $h_K\in [-\infty,\infty]^8$ in \eqref{hK}. 
If $v_K\in\operatorname{Vert}(\tilde\Delta)$ corresponds to a component $K\subset\widehat C_{\dd\Hy}$ then we set $r(v_K)=\rho_K$.
If $\rho_K>0$ then 
$\nu_\infty(K)\subset Q$, and we set $$(d_+(v_K),d_-(v_K))=[\nu_\infty(K)]\in H_2(Q)=\Z^2.$$
If $\rho_K=0$ then we set $$\delta(v_K)=[\nu_\infty(K)]\in H_3(\cp^3)=\Z.$$

We define the edges of $\tilde\Delta$ by means of the nodal points of $\widehat C_{\dd\Hy}$ as well as those marked points $p\in\widehat C_{\dd\Hy}$ which correspond to intersection points $\bar f_\alpha(\overline C_\alpha)\cap Q$ for large $t_\alpha$.
In the latter case the edge $E_p$ connects $v_p$ and $v_K$ if $p\in K$.
If $p\in \widehat C_{\dd\Hy}$ is a nodal point then the edge $E_p$ connects the vertices corresponding to the components adjacent to $p$.
It might happen that $E\in\operatorname{Edge}(\tilde\Delta)$ connects two vertices $v_1,v_2$ with $r(v_1)=r(v_2)$.
In such case we contract $E$ to the vertex $v_E$ of the new graph, and set $r(v_E)=r(v_1)$ 
and $d_\pm(v_E)=d_\pm(v_1)+d_\pm(v_2)$.
After performing all such contractions we get the graph $\Delta$.

We define $\phi(E_p)=f^{\dd\Hy}_\infty(p)$.
If $E_p\in\operatorname{Edge}(\Delta)$
we consider the vanishing circle $\gamma_p\subset \bar C_\alpha$ for large $t_\alpha$ as in the proof of  Proposition \ref{prop-vancircle}.
The image $\nu_\alpha(\gamma_p)$ is disjoint from $Q$ and converges to $\nu_\infty(p)\in Q$, since one of the adjacent vertices $v$ to $E$ has positive width $r(v)>0$.
We define $w(E_p)$ as the local linking number in the small 6-ball around $\nu_\infty(p)$ of $\nu_\alpha(\gamma_p)$ and $Q$.

The identity \eqref{total-deg}
holds since $d$ is the degree of the limiting curve $\nu_\infty:\widehat C_{\dd\Hy}\to\cp^3$.
If $v\in\operatorname{Vert}(\Delta)$ comes from a marked point then $v$ is 1-valent, and the condition \eqref{d+0} is vacuous.
If $v=v_K$ comes from a component $K\subset\widehat C_{\dd\Hy}$ then $\nu_\infty(K)\subset Q$ and $\pi_+(\nu_\infty(K))\in\cp^1$ is a point if $d_+(v)=0$, and thus the condition \eqref{d+0} also holds in this case.
To see that $\Delta$ is a $\Hy$-floor diagram we are left to verify \eqref{eq-div}.

Suppose that $v\in \operatorname{Vert}^+(\Delta)$ corresponds to $K\subset\widehat C_{\dd\Hy}$.
Here $K$ may be a single component or a a connected union of several components of the same width (resulting from the graph contraction $\tilde\Delta\to\Delta$).
For large $t_\alpha$ we define $K^\circ_\alpha\subset\widehat  C_{\dd\Hy}$ as the subsurface bounded by the vanishing cycles $\gamma_p\subset\widehat C_{\dd\Hy}$ corresponding to the edges $E_p\subset\Delta$ adjacent to $v$.
The image $\nu_\alpha(K^\circ_\alpha)$ is contained in a small neighborhood of $Q$ but disjoint from the quadric $Q\subset\cp^3$ itself.
Furthermore, the boundary components of $\nu_\alpha(\dd K^\circ_\alpha)$ corresponding to the edges $E_p$ are contained in small ball neighborhoods of $\nu_\infty(p)\in Q$ and correspond to $w(E_p)$ times the boundary of a small disk transversal to $Q$ in these balls.
The equation \eqref{eq-div} follows from the computation of the Euler class of the normal bundle of $Q$ in $\cp^3$ (equal to twice the plane section of $Q$).
If $v_0\in \operatorname{Vert}^0(\Delta)$ corresponds to $K\subset\widehat C_{\dd\Hy}$ then  \eqref{eq-div} follows since the intersection number of $Q$ and the result of attaching of small disks to $\nu_\alpha(\dd K^\circ_\alpha)$ in the corresponding small balls is twice the degree of $\nu_\infty(K)$.

We are left to prove the $\varkappa$-tropical convergence of $V_\alpha$ to $\Theta(\Delta)\subset\bHy$.
For a vertex $v\in \operatorname{Vert}(\Delta)$ corresponding to $K\subset\widehat C_{\dd\Hy}$ we form $K^\circ_\alpha$ as above.
Given a small $\epsilon>0$ set 
$$I_{v,\epsilon}=[r(v)-\epsilon,r(v)+\epsilon]$$ if $r(v)<\infty$ and
$I_{v,\epsilon}=[1/\epsilon,\infty]$ if $r(v)=\infty$.
For a large $t_\alpha$ we may assume 
$$\rho\circ\varkappa_{t_\alpha}(\gamma_p)\subset[0,\infty]\setminus I_{v,2\epsilon}$$
for the vanishing circles $\gamma_p$ adjacent to $v$.
The surfaces 
\[
K^\circ_\alpha(v;\epsilon)=K^\circ_\alpha\cap(\rho\circ\varkappa_{t_\alpha})^{-1}(I_{v,\epsilon})
\]
are smooth surfaces with boundaries corresponding to the edges of $\Delta$ adjacent to $v$,
and any critical point of $\rho\circ\varkappa_{t_\alpha}\circ\nu_\alpha:\overline C_\alpha\to (0,\infty)$
is contained in the surface $K^\circ_\alpha(v;\epsilon)$ for some $v$.
We have
$$\lim\limits_{t_\alpha\to\infty}K^\circ_\alpha(v;\epsilon)=\Theta(v)\cup\bigcup\limits_{E_+}R_{\phi(E_+)}[r(v),r(v)+\epsilon]\cup\bigcup\limits_{E_-}R_{\phi(E_-)}[r(v)-\epsilon,r(v)],
$$
where $E_+$ (resp. $E_-$) goes over edges connecting $v$ to vertices of higher (lower) width.
The complement $C_\alpha\setminus\bigcup\limits_v K^\circ_\alpha(v;\epsilon)$ consists of annuli $K_\alpha(E;\epsilon)$ for edges $E\subset\Delta$.
If $E$ connects vertices $v_1$ and $v_2$ with $r(v_1)< r(v_2)<\infty$ (resp. $r(v_1)< r(v_2)=\infty$) then $\varkappa_{t_\alpha}\circ\nu_\alpha(K_\alpha(E;\epsilon))$ converges to
$$\Theta_{[r(v_1)+\epsilon,r(v_2)-\epsilon]}(E)\
\text{(resp.  $\Theta_{[r(v_1)+\epsilon,1/\epsilon]}(E)$)},$$ 
since $\rho\circ\varkappa_{t_\alpha}\circ\nu_\alpha$ is critical point free on $K_\alpha(E;\epsilon)$.
Convergence of $\nu_\alpha(K_\alpha(E;\epsilon))$ to $q(E)=\nu_\infty(p)\in Q$ implies convergence of $H_{t_\alpha}(K_\alpha(E;\epsilon))$ to $$\pi_P^{-1}(q(E))\cap\varkappa^{-1}(\Theta_I(E)).$$

\end{proof}

\section{Amoebas of surfaces}
\subsection{Convexity of the complement}
Let $S\subset\cp^3$ be a surface. We may compare its hyperbolic amoeba $$\am_S=\am_{\Hy}(S)=\varkappa(S\setminus Q)\subset \Hy$$ against its conventional amoeba $$\am_{\R^3}(S)=\Log(S\cap (\C^\times)^3)\subset\R^3.$$ 
Each amoeba is closed in its ambient space. The complement $\R^3\setminus\am_{\R^3}(S)$ of the Euclidean amoeba is always non-empty, see \cite{GKZ}. But there exist surfaces $S$ such that $\am_{\Hy}(S)=\Hy$.

\begin{exa}\label{borel-ex}
The Borel subgroup $B\subset \I=\operatorname{PSL}_2(\C)$ consisting of upper-triangular matrices $\begin{pmatrix} a & b\\ 0 & d\end{pmatrix}$ coincides with $S\setminus Q$ for a coordinate plane in $\cp^3=\overline\I$.
It is generated by two 1-dimensional subgroups, 
\[
l_1=\{ \begin{pmatrix} 1 & s\\ 0 & 1\end{pmatrix}\ | \ s\in\R\}\ \ \ \text{and} \ \ \ l_2=\{ \begin{pmatrix} t & 0\\ 0 & t^{-1}\end{pmatrix}\ | \ t\in\R^\times\}.
\]
As we have seen in Section 3.1, the amoeba $\varkappa(l_1)$ is a horosphere while the amoeba $\varkappa(l_2)$ is a geodesic serving as the symmetry axis of the horosphere $\varkappa(l_1)$.
Thus the images of the products of elements of $l_1$ and $l_2$ under $\varkappa$ fill the entire space $\Hy$, i.e. we have $\am_B=\Hy$.
\end{exa}

\ignore{
Indeed, as we saw, a Borel subgroup is a semi-direct product of two lines. An amoeba of the first line is a horosphere and an amoeba of the second is a geodesic -- the axis of symmetry for the horosphere. 
The second line acts on $\Hy$ by translations along its amoeba. Thus, the amoeba of the Borel subgroup contains all the translations of the horosphere along the geodesic, and so coincides with the whole $\Hy.$

From this reasoning we can also conclude that a restriction of $\varkappa$ to a Borel subgroup is equivalent to taking the quotient by the rotations around a geodesic, so it is an honest circle fibration. In particular it is an evidence of certain stability for the property of amoeba to span the whole ambient space. We will show that this observation is absolutely true in degree one. 

In general an amoeba of a surface fills almost all ambient space. This is verified by the following simple lemma.
}

\begin{prop}
There exist surfaces $S$ such that $\am_S=\Hy$ as well as surfaces such that $\Hy\setminus\am_S\neq\emptyset$.
\end{prop}
\begin{proof}
In view of Example \ref{borel-ex}
it suffices to find a surface such that $0\notin\am_S$.
Recall that by \eqref{Pconj}, the inverse image $\varkappa^{-1}(0)$ can be identified with the real projective space $\rp^3\subset\cp^3$ after a suitable change of coordinates.
Thus any surface in $S\subset \cp^3$ with the empty real locus $S\cap\rp^3=\emptyset$, e.g. $\{ z_0^2+z_1^2+z_2^2+z_3^2=0\}$ provides a required example after a change of coordinates. 
\end{proof}

\begin{prop}
For any surface $S\subset\cp^3$ we have $\overline\varkappa(S)\supset\dd\Hy$. 
\end{prop}
\begin{proof}
Recall that $\overline\varkappa(Q)=\dd\Hy$.
If $Q\subset S$ the proposition is immediate.
Otherwise, the intersection $S\subset Q$ is a curve of bidegree $(d,d)$, where $d$ is the degree of $S$.
We have $\overline\varkappa(S\cap Q)=\pi_+(S\cap Q)=\dd\Hy$.
\end{proof}

It is interesting to note that there exist surfaces $S\neq Q$ such that $\Hy\setminus\am_S$ is unbounded.
\begin{exa}
Choose a point $x\in\dd\Hy$.
The inverse image $\overline\varkappa^{-1}(x)$ is a line on the quadric $Q$.
Let $S_x\subset \cp^3$ be the surface obtained from $Q$ by a small rotation around the line $\overline\varkappa^{-1}(x)$, i.e. a quadric $S_x$, such that $S_x\cap Q$ consists of this line and an irreducible curve of bidegree $(1,2)$ (on either of the two quadrics).
\begin{prop}\label{unbound}
The complement $\Hy\setminus \am_{S_x}$ in this example is unbounded. The geodesic ray $R_x$ 
emanating from the origin $0\in\Hy$ in the direction of $x\in\dd\Hy$ is disjoint from the amoeba $\am_{S_x}$.
\end{prop}
\begin{proof}
Recall that $P=\varkappa^{-1}(0)$ is the fixed locus of the antiholomorphic involution \eqref{Pconj} after a coordinate change.
For each $z\in\overline\varkappa^{-1}(x)$ there exists a unique $P$-real line $L_z\subset\cp^3$ passing through $z$ and the image of $z$ under \eqref{Pconj}.
The real locus $\R L_z=L_z\cap P$ divides the line $L_z$ into two half-spheres.
Denote with $H_z$
the closure of
the half-sphere containing $z$.
By Proposition \ref{l-act} and Lemma \ref{lr-act}, we have
$$
\varkappa^{-1}(R_x
)=\bigcup\limits_{z\in\varkappa^{-1}(x)} H_z\setminus Q.
$$

Since $H_z$ and $Q$ intersect at $z\in Q\cap S_x=\varkappa^{-1}(R_x)$ transversally, we have $H_z\cap S_x=\{z\}$. Thus $H_z\setminus Q$ is disjoint from $S_x$.  
\end{proof}
\end{exa}
 
\begin{theorem}\label{t_odd}
If $S\subset\cp^3$ is a surface of odd degree then $\am_S=\Hy$.
\end{theorem}
\begin{proof}
Note that $\varkappa^{-1}(x)$, $x\in\Hy$, is isotopic to $\varkappa^{-1}(0)$
which in its turn can be identified with $\rp^3\subset\cp^3$.
Thus $H^2(\varkappa^{-1}(x);\Z_2)=\Z_2$ and the restriction homomorphism 
$$H^2(\cp^3;\Z_2)\to H^2(\varkappa^{-1}(x);\Z_2)$$
is an isomorphism.
Non-vanishing of the class in $H^2(\cp^3;\Z_2)$ which is Poincar\'e dual to $[S]=d[\cp^2]\neq 0\in H_4(\cp^3;\Z_2)$ implies that $S\cap\varkappa^{-1}(x)\neq\emptyset$.
\end{proof}

In the general case we have the following theorem. 
Recall that a convex set $K\subset\Hy$ is a set such that $x,y\in K$ implies that the geodesic segment $[x,y]$ between $x$ and $y$ is also contained in $K$.
\begin{theorem}\label{t_even}
For any surface $S\subset\cp^3$ different from the quadric $Q=\{(z_0:z_1:z_2:z_3)\ |\ z_0z_3-z_1z_2\}$  the complement $\Hy\setminus\am_S$
is an open convex set in $\Hy$. In particular, it is connected.
\end{theorem}  
\begin{proof}
\ignore{
Note that we may assume that $S\cap Q\neq Q$.
Indeed, otherwise we may pass to a surface $S'$ defined as the closure of $S\setminus Q$ in $\cp^3$. The hypothesis $S\neq Q$ implies that $S'\neq\emptyset$ while we have $\am_{S'}=\am_S$ by construction.
Thus the intersection $S\subset Q$ is a curve of bidegree $(d,d)$, where $d$ is the degree of $S$ and $\overline\varkappa(S\cap Q)=\pi_+(S\cap Q)=\dd\Hy$.
Furthermore, any ..
}
Openness of $\Hy\setminus\am_S$ is implied by Corollary \ref{closed-set}.
Let $x,y\in \Hy\setminus\am_S$ be distinct points.
Each oriented geodesic $\gamma$ is an amoeba of a line $L_\gamma\subset\cp^3$ by Proposition \ref{line-descr}.
If $x\in\gamma$ then $\varkappa^{-1}(x)\cap L_\gamma$ is a circle dividing the line $L_\gamma$ to two half-spheres $H^+_\gamma$ and $H^-_\gamma$. 
The intersection numbers $r^+_x$ (resp. $r^-_x$) of $S$ and $H^+_\gamma$
(resp. of $S$ and $H^-_\gamma$) do not depend on the choice of $L_\gamma$ since $x\notin \am_S$.
Since we may continuously 
deform $\gamma$ to itself with the reversed orientation via rotations around $x$, we have $r^+_x=r^-_x$. 
Similarly, we have $r^+_y=r^-_y$. Furthermore, $r^+_x+r^-_x=r^+_y+r^-_y$ since both quantities coincide with the degree $d$ of $S$ in $\cp^3$.
Thus 
\begin{equation}\label{dsur2}
r^+_x=r^-_x=r^+_y=r^-_y=d/2.
\end{equation}

Let $\gamma_{[x,y]}\subset\Hy$ be the geodesic passing through $x$ and $y$ and $L$ be a line such that $\gamma_{[x,y]}=\am_L$. The complement $L\setminus \varkappa^{-1}(\{x,y\})$ consists of two half-spheres and a cylinder $Z$ with $\varkappa(Z)=[x,y]$.
Since the intersection number of $S$ and each of the two half-spheres is $d/2$ by \eqref{dsur2}, the intersection number of the holomorphic cylinder $Z$ and the holomorphic surface $S$ is zero, thus they are disjoint.
Lemma \ref{lr-act} and Proposition \ref{line-descr} imply that $\varkappa^{-1}([x,y])$ is fibered by such holomorphic cylinders.
Thus $[x,y]\cap\am_S=\emptyset$.
\end{proof}
Note that the identity \eqref{dsur2} provides another proof of Theorem \ref{t_odd}.

\subsection{Left $\operatorname{PSL}_2$-Gauss map}
Multiplication from the left defines a holomorphic action of $\I$ on itself, and thus also a trivialization of the tangent bundle $T\I\approx\I\times\C^3$.
Let $S\subset\cp^3$ be a complex surface and $S^\circ\subset S$ be the smooth locus of $S\cap\I$.
We define the left $\operatorname{PSL}_2$-Gauss map
\begin{equation}
\gamma^-_S:S^\circ\to \cp^2
\end{equation}
as the map sending $A\in S^\circ$ to $A^{-1}T_AS$, i.e. to the tangent space to the surface $S$ after bringing it to the unit element $e\in\I$ by left multiplication.
In this way we get a plane in the tangent space $T_e\I$ which can be viewed as an element of ${\mathbb{P}}(T^*_e\I)=\cp^2$.
Similarly we can define the right $\operatorname{PSL}_2$-Gauss map $\gamma^+_S:S^\circ\to \cp^2$.

It is easy to compute $\gamma_S^-$ in terms of the homogeneous polynomial $p$ defining $S=\{(a:b:c:d)\ |\ p(a,b,c,d)=0\}$.
For this purpose it suffices to choose three linearly independent tangent vectors in $T_e\I$.
It is convenient to choose them tangent to $P=\varkappa^{-1}(0)$, i.e. in the real part of the Lie algebra of $\operatorname{SO}(3)=P$ which coincides with that of $\operatorname{SU}(2)$.
We can take
$
\begin{pmatrix}
i & 0\\
0 & -i
\end{pmatrix}
$,
$
\begin{pmatrix}
0 & 1\\
-1 & 0
\end{pmatrix}
$
and
$
\begin{pmatrix}
0 & i\\
i & 0
\end{pmatrix}
$
for the basis of $T_eP$.
Taking partial derivatives of $p$ with respect to images of these vectors under the differential of the left multiplication by 
$A=
\begin{pmatrix}
a & b\\
c & d
\end{pmatrix}
\in S$
we get the following expression for $\gamma_S^-(A)$ in coordinates:
\[
( i(ap_a-bp_b+cp_c-dp_d) : 
-bp_a+ap_b-dp_c+cp_d : 
i(bp_a+ap_b+dp_c+cp_d) )\in\cp^2,
\]
where $p_a,p_b,p_c,p_d$ stand for the corresponding partial derivatives of $p$ at $A$.
Thus it is an an algebraic map on $S^\circ$ whose degree cannot be greater than $d^3$.

\begin{prop}\label{prop-N}
Let $N\subset \cp^3$ be a plane tangent to the quadric $Q$.
Then $\gamma^-_N$ is a constant map.
\end{prop}
\begin{proof}
Since $N$ is tangent to $Q$ the intersection $Q\cap N$ is a union of two lines.
By Lemma \ref{lr-act} the image $BN\subset \cp^3$, $B\in\I$, is a plane containing one of these lines.
Thus either $N=BN$ or $N\cap BN\setminus Q=\emptyset$. Thus a vector tangent to $N$ at $A\in N$ must remain tangent to $N$ under the left multiplication by $B$ if $BA\in N$.
\end{proof}

\begin{prop}
Let $R\subset \cp^3$ be a plane not tangent to the quadric $Q$.
Then $\gamma^-_R$ is a map of degree 1.
\end{prop}
\begin{proof}
Since the coordinate $\gamma_R^-$ are given by linear functions for the plane $R$, it suffices to prove that $\gamma_R^-$ is not constant if $R$ is transversal to $Q$.
Assuming the contrary, for any line $l\subset R$ the image of $l$ under the Gauss map $\gamma_-$ would be contained in the line corresponding to the image of $\gamma_S^-(R)$ in $\cp^2$, see Remark \ref{rem-Sym}.
But if $R\cap Q$ is a smooth curve of bidegree $(2,2)$ then any pair of points in $\Sym^2(\cp^1)=\cp^2$ appears in this image for one of the lines as we may choose a pair of points on $R\cap Q$ with arbitrary images under $\pi_-$.
\end{proof}

The map \eqref{Pconj} 
leaves $P=\varkappa^{-1}(0)$ fixed.
Thus it defines the real structure on $T_e\I=\cp^2$ such that $T_e P=\rp^2$.
We refer to this real structure as the $P$-real structure on $\cp^2$.

\begin{theorem}\label{thm-Gaussleft}
Let $S\subset\cp^3$ be a surface and $z\in S$ be its smooth point.
Then $z$ is a critical point for $\varkappa|_S$ if and only if $\gamma_S^-(z)\in\rp^2$,
i.e. $\gamma_S^-(z)$ is a $P$-real point.
\end{theorem}
\begin{proof}
Note that a plane $R\subset\cp^2$ always intersect $P=\rp^3$ transversally.
It is $P$-real if and only if $R\cap P$ is real 2-dimensional. Otherwise, $R\cap P$ is 1-dimensional.
If $z=e\in S$ then $e$ is critical for $\varkappa|_S$ if and only if the dimension of the kernel of the differential of $\varkappa$ is greater than one, i.e. if $T_eS$ is $P$-real.
Left multiplication by $z$ extends this argument for an arbitrary smooth point $z\in S$.
\end{proof}

Consider the locus $C_N$ consisting of all points $z\in \cp^2$ such that there exists a plane $N\subset\cp^3$ tangent to $Q$ such that $\gamma_N^-(N)=\{z\}$ as in Proposition \ref{prop-N}.
\begin{prop}
The locus $C_N\subset\cp^2$ is a non-degenerate conic curve invariant with respect to the $P$-real structure without $P$-real points. 
\end{prop}
\begin{proof}
The locus $C_N$ is parametrized by one of the factors in $Q=\cp^1\times\cp^1$ and thus is a holomorphic curve. It must intersects any pencil of planes through $e\in\I$ at two points since $Q$ is a quadric, and so is its projective dual surface.
If $N\subset\cp^3$ were a plane such that $\gamma^-_N$ is constant and real then by Theorem \ref{thm-Gaussleft} its amoeba would have to be 2-dimensional.
Therefore this assumption is in contradiction with Theorem \ref{t_odd}.
\end{proof}

The locus $C_N$ has the following special property with respect to the left Gauss map in $\I$.
\begin{prop}
Let $w\in C_N$ and $S\subset\cp^3$ be a generic surface of degree $d$.
Then $\gamma_S^-(w)$ consists of $d(d-1)^2$ points.
\end{prop}
\begin{proof}
Let $N_w$ be a plane tangent to $Q$ such that $\gamma_{N_w}^-(N_w)=\{w\}$.
As in the proof of Proposition \ref{prop-N}, Lemma \ref{lr-act} implies that we have $w=\gamma_S^-(z)$ for a point $z\in S$ if and only if $S$ is tangent at $z$ to a surface $BN_z$ for some $B\in\I$. By the hypothesis, $S$ is generic and thus there exist $d(d-1)^2$ planes passing through the line $\overline\varkappa^{-1}(w)$ and tangent to $S$ as $d(d-1)^2$ is the class of the smooth surface of degree $d$.
\end{proof}

In particular, in the case  of $d=1$, the locus $\gamma_S^-(w)$ is empty for $w\in C_N$ even for the case when $S$ is a generic plane and the degree of $\gamma_S^-$ is one.

 
 \ignore{
 \section{Limits and Spherical Coamoebas}
Note that so far we were using only the right quotient map $\varkappa:\I\rightarrow\I\slash \operatorname{SO}(3).$ Since $\I$ is non-commutative it make sense to consider the right quotient map together with the left quotient map $$\hat\varkappa(A)=(A(O),A^{-1}(O))$$ which sends $A\in\I$ to $\hat H=\{(p,q)\in\Hy\times\Hy| |pO|=|qO|\}.$  We view $\hat H$ as a cone over $Q$, i.e. $[0,\infty)\times Q\slash (\{0\}\times Q) .$  Denote by $\hat R_t$ the rescaling on $\hat H.$

theorem on limits of constant families

coamoebas of curves

conjectures
}

\bibliography{b}
\bibliographystyle{plain}

\end{document}